\documentclass[a4paper]{article}

\usepackage{a4wide}
\usepackage[UKenglish]{babel}
\usepackage{amssymb}
\usepackage[inline]{enumitem}
\usepackage{graphicx}
\usepackage{subfigure}
\usepackage{color}
\usepackage[colorlinks=true,linkcolor=blue,citecolor=red]{hyperref}
\usepackage{amsmath,amsfonts,amsthm}

\usepackage{multirow}

\newtheorem{rem}{Remark}

\newtheorem{prop}{Proposition}
\newtheorem{thm}{Theorem}

\def\be{\begin{equation}}
\def\ee{\end{equation}}
\def\bea{\begin{eqnarray}}
\def\eea{\end{eqnarray}}

\def\RR{\mathbb R}

\begin{document}

\title{Structure preserving stochastic Galerkin methods for Fokker-Planck equations with background interactions}

\author{Mattia Zanella \\
{\small Department of Mathematical Sciences ``G. L. Lagrange''} \\
		{\small Dipartimento di Eccellenza 2018-2022} \\
		{\small Politecnico di Torino, Torino, Italy} \\
		{\small\tt mattia.zanella@polito.it}}

\date{}


\maketitle
\begin{abstract}
This paper is devoted to the construction of structure preserving stochastic Galerkin schemes for Fokker-Planck type equations with uncertainties and interacting with an external distribution, that we refer to as  a background distribution. The proposed methods are capable to preserve physical properties in the approximation of statistical moments of the problem like nonnegativity, entropy dissipation and asymptotic behaviour of the expected solution. The introduced methods are second order accurate in the transient regimes and high order for large times. We present applications of the developed schemes to the case of fixed and dynamic background distribution for models of collective behaviour. 

\medskip


\end{abstract}

\section{Introduction}
\label{sec1}

Uncertainty quantification (UQ) for partial differential equations describing real world phenomena gained an increased interest in recent years \cite{CPZ,DPL,DPZ,HJJ,HJ,HJX,TZ,X}. One of the main advantages of UQ methods relies in its capability to provide a sound mathematical framework to replicate realistic experiments. The introduction of stochastic parameters reflects our incomplete information on the initial configuration of a system, on its inner interactions forces and on the modelling parameters as well. Methods and ideas of UQ had a deep impact on the applied community in terms of effectivity for a robust conception, modelling and sensitivity analysis of the problem of interest.  

In the context of kinetic equations, this issue can be translated on a general uncertainty affecting a distribution function of particles/agents, whose evolution is influenced by the presence of a random variable $\theta$, taking value in the set $I_{\Theta}\subseteq \RR$, and with known probability distribution function $\Psi(\theta): I_{\Theta}\rightarrow \RR^+$. In particular, in the present manuscript we are interested in Fokker-Planck type equations for the evolution of the distribution $f = f(\theta,v,t)$, $v\in V\subseteq \RR^{d_v}$, $\theta\in I_{\Theta}$ and $t\ge 0$ is the time. The introduced distribution represents the proportion of particles/agents in $[v,v+dv]$ at time $t\ge 0$ and for given value of uncertainty $\theta\in I_{\Theta}$. In more details, we consider the partial differential equation
\begin{equation}\label{eq:FP_general}
\partial_t f(\theta,v,t) = \nabla_v \cdot \left[  \mathcal B[g](v,t)f(\theta,v,t)+\nabla_v (D(v)f(\theta,v,t)) \right],
\end{equation}
where $v\in V \subseteq\RR^{d_v}$ and $\mathcal B[\cdot]$ is the operator
\be\label{eq:B}
\mathcal B[g] (v,t) = \int_{V} P(v,v_*) (v-v_*) g(v_*,t)dv_*,
\ee
where $g = g(v,t)$ is a background distribution, whose dynamics do not incorporate the presence of the uncertain quantity $\theta\in I_{\Theta}$. In applications, problems with background interactions are very often considered to mimic the influence of environmental factors on the agents' dynamics especially in socio-economic and life sciences. For example, the process of knowledge formation depends on social factors that determine the progress in competence acquisition of individuals, see \cite{PT2,PVZ} and the references therein. Similarly, in soft-matter physics biological particles like cells undergo various heterogeneous stimuli forcing their observable motion \cite{PreziosiTosin}. Other examples have been studied in opinion dynamics, economic processes for the formation of wealth distributions, and urban growth theory, see \cite{GT} for a review. 

We consider for \eqref{eq:FP_general} an uncertain initial distribution $f(\theta,v,0)$, no-flux boundary conditions are considered on the boundaries of the domain to enforce conservation of the total mass of the system. A clear understanding on the global behavior of the system governed by \eqref{eq:FP_general}-\eqref{eq:B} is obtained in terms of expected statistical quantities whose accurate and physically admissible description is therefore of paramount importance. 

Due to the increased dimensionality of the problem induced by the presence of uncertainties, the issue of developing fast converging numerical methods for the approximation of statistical quantities is of the highest importance. Among the most popular numerical methods for the UQ, stochastic Galerkin (SG) methods gained in recent years increasing interest since they provide spectral convergence in the random space under suitable regularity assumptions \cite{APZa,JXZ,X,XK,ZJ}. Despite their accuracy, direct application of SG techniques to high dimensional random vectors often suffers of the so-called curse of dimensionality and require ad-hoc solutions, see e.g. \cite{SHJ}.  Similarly to classical spectral methods SG methods generally require a strong modification of the original problem and can lead to the loss of structural properties like positivity of the solution, entropy dissipation and hyperbolicity, when applied to hyperbolic and kinetic equations, see \cite{DPL,JP}. The loss of structural properties of the solution induces an evident gap in its true physical meaning. To overcome this problem, recently has been proposed a novel methods that combines both Monte Carlo and SG generalized polynomial chaos methods (MCgPC) and which preserves spectral accuracy in the random space. In particular, MCgPC methods can mitigate the curse of dimensionality induced by SG-gPC. We refer to \cite{CZ,CPZ} for a detailed discussion on these methods.

In the present manuscript we construct structure preserving methods for the SG formulation of the problem in the case of background interactions. In order to do that we will take advantages of structure preserving (SP) methods \cite{PZ1,PZ2}, that have been designed to preserve the mentioned structural properties of the solution of nonlinear Fokker-Planck equations without restriction on the mesh size. We consider applications of the developed schemes both in case of fixed and dynamic background. 

The rest of the paper is organized as follows. In Section \ref{sect:2} we briefly introduce stochastic Galerkin methods for the problem of interest where the interactions take place with respect to a deterministic background, stability results are proved and discussed together with the analysis of trends to asymptotic states. In Section \ref{sect:structure} we derive structure preserving methods in the Galerkin setting, positivity conditions for explicit and semi-implicit schemes are discussed and we prove entropy inequality for a class of one dimensional Fokker-Planck models. Several applications of the schemes are finally considered in Section \ref{sect:num} for several problems arising in the description of collective phenomena in socio-economic and life-sciences. Some conclusions are reported at the end of the manuscript. 

\section{Stochastic Galerkin methods for kinetic equations}\label{sect:2}
 For simplicity of presentation we consider the case $d_v=1$. We focus on real-valued distributions depending on a one dimensional random input. Let $(\Omega,F,P)$ be a probability space where as usual $\Omega$ is the sample space, $F$ is a $\sigma-$algebra and $P$ a probability measure, and let us defined a random variable 
\[
\theta: (\Omega,F) \rightarrow (I_{\Theta},\mathcal B_{\mathbb R}),
\]
with $I_\Theta\subset \mathbb R$ and $\mathcal B_{\mathbb R}$ is the Borel set. We focus on real-valued distributions of the form $f(\theta,v,t): \Omega \times V \times [0,T] \rightarrow \mathbb R^d$. In the present section we derive a stochastic Galerkin approximation for Fokker-Planck equation with uncertain initial distribution and background interactions \eqref{eq:FP_general}. 

Let us consider the linear space $\mathbb P_M$ of polynomials of degree up to $M$ generated by a family of orthogonal polynomials $\{\Phi_h(\theta) \}_{h=0}^M$ such that
\[
\mathbb E[\Phi_h(\theta)\Phi_k(\theta)] = \int_{I_{\Theta}} \Phi_h(\theta)\Phi_k(\theta)\Psi(\theta)d\theta = \|\Phi_h^2(\theta) \|_{L^2(\Omega)}\delta_{hk},
\]
being $\delta_{hk}$ the Kronecker delta function. Assuming that $\Psi(\theta)$ has finite second order moment we can approximate the distribution $f \in L^2(\Omega,\mathcal F,P)$ in terms of the following chaos expansion
\be\label{eq:fM}
f(\theta,v,t) \approx f^M(\theta,v,t) = \sum_{k=0}^M \hat f_k(v,t)\Phi_k(\theta),
\ee
being $\hat f_k(v,t)$ the projection of $f$ into the polynomial space of degree $k$, i.e.
\[
\hat{f}_k(v,t) = \mathbb E[f(\theta,v,t)\Phi_k(\theta)], \qquad k = 0,\dots,M.
\]
Plugging $f^M$ into \eqref{eq:FP_general} we obtain 
\be\label{eq:FP_M}
\partial_t f^M(\theta,v,t) =\partial_v \left[  \mathcal B[g](v,t)f^M(\theta,v,t) + \partial_v (D(v) f^M(\theta,v,t)) \right].
\ee
Hence, by multiplying \eqref{eq:FP_M} by $\Phi_h(\theta)$ for all $h=0,\dots,M$ and after projection in each polynomial space we obtain the following system of $M+1$ deterministic kinetic-type PDEs 
\be\label{eq:FP_h}
\partial_t \hat f_h(v,t) = \partial_v\left[ \mathcal B[g](v,t) \hat f_h(v,t) + \partial_v (D(v)\hat f_h(v,t))\right],
\ee
with the initial conditions
\[
\hat f_h(v,0) = \mathbb E[f(\theta,v,0)\Phi_h(\theta)]. 
\]
The related deterministic subproblems can be tackled through suitable numerical methods and the approximation of statistical quantities of interest are defined in terms of the projections. In particular we have
\be\label{eq:expected}
\mathbb E[f(\theta,v,t)]\approx \hat f_0(v,t), 
\ee
whose evolution is given by \eqref{eq:FP_h} in the case $h=0$. Thanks to the orthogonality in $L^2(\Omega)$ of the polynomials $\{\Phi_h\}_{h=0}^M$ we have
\[
\mathbb E[f(\theta,v,t)^2]- \mathbb E[f(\theta,v,t)]^2 \approx \mathbb E[(f^M(\theta,v,t))^2]- \mathbb E[f^M(\theta,v,t)]^2,
\]
and from \eqref{eq:fM} we have
\[
 \mathbb E\left[ \sum_{k=0}^M \hat f_k^2(v,t)\Phi_k^2(\theta) + 2\sum_{k=0}^M \sum_{h = 0}^{k-1} \hat f_k(v,t)\hat f_h(v,t) \Phi_k(\theta)\Phi_h(\theta)\right] - \hat f_0^2(v,t).
\]
Therefore the variance of the solution is approximated in terms of the projections as follows
\be\label{eq:var}
\textrm{Var}[f(\theta,v,t)] \approx \sum_{k = 0}^M \hat f_h^2(v,t) \mathbb E[\Phi_k^2] - \hat f_0^2(v,t). 
\ee

We observe that the initial mass defined by $\int_V \hat f_h(v,0)dv $ is conserved in time assuming no-flux boundary conditions, i.e.
\[
\mathcal B[g](v,t) + \partial_v D(v) = 0,\qquad v\in\partial V. 
\]

Let us introduce the vector $\hat{\textbf{f}}(v,t) = \left( \hat f_0(v,t),\dots,\hat f_M(v,t)\right)$. If we define as $\|\hat{\textbf{f}}(v,t) \|_{L^2}$ the standard $L^2$ norm of the vector $\hat{\textbf{f}}(v,t)$
\[
\|\hat{\textbf{f}}(v,t) \|_{L^2} = \left[ \int_{V} \left(\sum_{h=0}^M \hat f_h^2(v,t) \right)dv\right]^{1/2},
\]
then from the orthonormality of the introduced basis $\{\Phi_h\}_{h=0}^M$ in $L^2(\Omega)$ we have that 
\[
\| f^M(\theta,v,t) \|_{L^2(\Omega)} = \|\hat{\textbf{f}}(v,t) \|_{L^2},
\]
 where
\[
\|f^M(\theta,v,t)\|_{L^2(\Omega)} = \left(\int_{I_\Theta} \int_{V}\left( \sum_{h=0}^M \hat f_h(v,t) \Phi_h(\theta)\right)^2dv \Psi(\theta)d\theta\right)^{1/2}.
\]
We can reformulate the problem \eqref{eq:FP_h} in a more compact form as follows
\[
\partial_t \hat{\textbf{f}}  = \partial_v \left[ \textbf{B}\hat{\textbf{f}} + \textbf{D}\partial_v \hat{\textbf{f}}\right],
\]
where $\textbf{B} = \{B_{ij}\}_{i,j=1}^{M+1}$ and $\textbf{D}=\{D_{i,j}\}_{i,j=1}^N$ are diagonal matrices with components
\[
\begin{split}
\textbf{B}_{i,i} &= \mathcal B[g](v,t) + \partial_v D(v),   \qquad \textbf{B}_{i,j} = 0\\
\textbf{D}_{i,i} &= D(v), \qquad\qquad\qquad\qquad\,\textbf{D}_{i,j}  = 0.
\end{split}\]
The following stability result can be established
\begin{thm}
If $\| \partial_v \mathcal B[g](v,t)\|_{L^\infty} \le C_{\mathcal B}$, with $C_{\mathcal B}>0$, and if $D\le C_D$ we have for all times $t>0$
\[
\|\hat{\normalfont{\textbf{f}}}(v,t) \|_{L^2}^2 \le e^{t(C_{\mathcal B} + 2C_D)}\|\hat{\normalfont{\textbf{f}}}(v,0) \|_{L^2}^2,
\]
provided that the boundary terms vanish for all $h=0,\dots,M$
\[
\hat{f}^2_h(v,t)\mathcal B[g](v,t) \Big|_{v\in\partial V} = 0, \quad \hat{f}_h(v,t)\partial_v(D(v)\hat{f}_h(v,t)) \Big|_{v\in\partial V}  =0, \quad \hat{f}_h(v,t)D(v)\partial_v \hat{f}_h(v,t) \Big|_{v\in\partial V}  = 0
\]
\end{thm}
\begin{proof}
We multiply \eqref{eq:FP_h} by $\hat f_h(v,t)$ and integrate over $V\subseteq \mathbb R$
\begin{equation}\begin{split}\label{eq:th1}
\int_V \partial_t \left(\dfrac{1}{2}\hat f_h^2(v,t) \right)dv &= \int_V \hat f_h(v,t) \partial_v  \left[\mathcal B[g](v,t)\hat f_h(v,t) + \partial_v (D(v)\hat f_h(v,t)) \right] dv  \\
& =  \underbrace{\int_V \hat{f}_h(v,t) \partial_v \left( \mathcal B[g](v,t) \hat{f}_h(v,t)\right)dv}_{A} + \underbrace{\int_V \hat{f}_h(v,t) \partial_v^2 \left(D(v)\hat{f}_h(v,t)\right)dv}_{B}
\end{split}\end{equation}
The integral $A$ may be rewritten as follows by direct computation
\[
\int_V \hat f_h(v,t) \partial_v  \left(\mathcal B[g](v,t)\hat f_h(v,t)\right)dv = \int_V \hat{f}_h(v,t) \mathcal B[g](v,t)\partial_v \hat f_h(v,t)dv + \int_V \hat{f}_h^2(v,t) \partial_v \mathcal B[g](v,t)dv
\]
and, integrating by parts, it is equivalent to 
\[\begin{split}
 &\int_V \hat f_h(v,t) \partial_v  \left(\mathcal B[g](v,t)\hat f_h(v,t)\right)dv =  \\
 &\qquad \qquad - \int_V \hat f_h(v,t) \partial_v \left(\mathcal B[g](v,t)\hat f_h(v,t)\right)dv+  \int_V \hat f_h^2(v,t)\partial_v \mathcal B[g](v,t)dv 
\end{split}\]
being $\hat{f}^2_h(v,t)\mathcal B[g](v,t) \Big|_{v\in\partial V} = 0$. Therefore, the integral $A$ may be written as
\[
\int_V \hat{f}_h(v,t)\partial_v \left(\mathcal B[g](v,t)\hat{f}_h(v,t) \right)dv = \dfrac{1}{2}\int_V \hat{f}_h^2 (v,t) \partial_v \mathcal B[g](v,t)dv.
\]
Hence, the following estimate holds
\[
\sum_{h=0}^M\int_V \hat f_h(v,t) \partial_v  \left(\mathcal B[g](v,t)\hat f_h(v,t)\right)dv = 
\dfrac{1}{2}\sum_{h=0}^M \int_V \hat f_h^2(v,t)\partial_v \mathcal B[g](v,t)dv \le \dfrac{C_{\mathcal B}}{2} \| \hat{\textbf{f}}(v,t) \|^2_{L^2}.
\]
Furthermore, for the integral $B$ in \eqref{eq:th1} we have
\[
\begin{split}
\int_V \hat{f}_h(v,t) \partial_v^2 \left(D(v)\hat f_h(v,t) \right)dv &= \int_V \left(\partial_v^2 \hat f_h(v,t)\right)D(v)\hat f_h(v,t)dv \\
&\le -C_D \int_V \left(\partial_v \hat f_h(v,t)\right)^2,
\end{split}
\]
provided $\hat{f}_h(v,t)\partial_v(D(v)\hat{f}_h(v,t)) \Big|_{v\in\partial V}  =0$, and $\hat{f}_h(v,t)D(v)\partial_v \hat{f}_h(v,t) \Big|_{v\in\partial V}  = 0$.

Finally, after summation on $h=0,\dots,M$ of the obtained bounds, we get
\[
\begin{split}
\dfrac{1}{2}\|\hat{\textbf{f}}(v,t) \|_{L^2}^2& \le \dfrac{C_{\mathcal B}}{2}\|\hat{\normalfont{\textbf{f}}}(v,t)  \|_{L^2}^2   - \| \partial_v \hat{\normalfont{\textbf{f}}}(v,t) \|_{L^2}^2 \\
& \le \left(\dfrac{C_{\mathcal B}}{2} + C_D\right)\| \hat{\textbf{f}} \|_{L^2},
\end{split}\]
and thanks to the Gronwall's theorem we can conclude. 
\end{proof}

\begin{rem}
The background distribution $g(v,t)$ is in general ruled by an additional PDE that does not depend on the stochastic density function $f(\theta,v,t)$ and does not incorporate additional uncertainties. In the case of evolving background we need to couple to \eqref{eq:FP_general} its dynamics. 
\end{rem}

\subsection{Asymptotic behaviour}\label{subsect:asymp}
Under suitable smoothness assumptions the introduced Fokker-Planck equation has a unique smooth solution, see e.g. \cite{Pav,Risken}. In the present section we concentrate on the large time solution of the introduced class of problems, known in the literature as equilibrium solutions or steady states. Assuming that the dynamics of the background $g(v,t)$ admit a unique stationary state the asymptotic distribution of \eqref{eq:FP_general} is solution of the differential equation
\begin{equation}\label{eq:cond_finf}
\mathcal B[g^\infty](v) f^\infty(\theta,v) + \partial_v (D(v)f^\infty(\theta,v)) = 0,
\end{equation}
which gives
\[
\dfrac{\partial_v f^\infty(\theta,v)}{f^\infty(\theta,v)} =  - \dfrac{\mathcal B[g^\infty](v) + D^\prime(v)}{D(v)},
\]
and therefore the analytical stationary distribution of the original problem reads
\be\label{eq:quasi}
f^\infty(\theta,v) = C(\theta) \exp \left\{ - \int \dfrac{\mathcal B[g^\infty](v) + D^\prime(v)}{D(v)}dv\right\},
\ee
being $C(\theta)>0$ a normalization constant depending only on the initial uncertainties of the problem. \\
On the other hand, the asymptotic solutions $f^\infty_h(v)$ of \eqref{eq:FP_h} in each polynomial space of degree $h=0,\dots,M$ are defined by solving the following set of differential equations
\be\label{eq:FP_h_stat}
\mathcal B[g^\infty](v)\hat{f}^\infty_h(v) + \partial_v (D(v)\hat{f}^\infty_h(v)) = 0, \qquad h = 0,\dots,M,
\ee
whose stationary states are
\be\label{eq:quasih}
\hat f^\infty_h(v) = C_h \exp\left\{- \int \dfrac{\mathcal B[g^\infty](v) +D^\prime(v)}{D(v)}dv \right\}
\ee
being $C_h$ such that 
\[
\int_{I_{\Theta}} f^\infty(\theta,v)\Phi_h(\theta)\Psi(\theta)d\theta = \hat{f}^\infty_h(v). 
\]
We can observe how, if the initial state has deterministic mass $\int_V f(\theta,v,0)dv = \bar \rho>0$, the asymptotic state of the problem given by \eqref{eq:quasi} does not incorporate any uncertainty. In particular, this is induced by the fact that the normalization constant does not depend anymore on the uncertainty of the problem, i.e. $C(\theta) = \bar C$ for all $\theta\in I_{\Theta}$. This fact reflects on the asymptotic state of each projection $\hat f_h^\infty(v)$, $h=0,\dots,M$, since $\mathbb E[\bar C\Phi_h(\theta)]=0$ for $h>0$. Therefore, in the case of deterministic initial mass we obtain 
\[
\hat f_h^\infty(v) = 
\begin{cases}
\bar C \exp\left\{- \displaystyle \int \dfrac{\mathcal B[g^\infty](v) +D^\prime(v)}{D(v)}dv \right\} & \textrm{if}\quad h = 0\\
0 & \textrm{if}\quad h>0,
\end{cases}
\]
and the variance of $f(\theta,v,t)$ vanishes asymptotically. In the general case of uncertain initial mass the asymptotic state still depends on $\theta\in I_\Theta$. \\

In the following we explicit the trend to equilibrium defined by stochastic background interaction models following the ideas in \cite{FPTT}.

\subsubsection{Constant background}\label{subsect:constant}
Let us assume that the background is fixed, i.e. $\mathcal B[g](v,t) = \mathcal B[g](v)$. In particular, the stationary distribution of  \eqref{eq:FP_general} is solution of the following differential equation  
\begin{equation}\label{eq:cond_gcons}
\mathcal B[g](v)f^\infty(\theta,v) + \partial_v\left( D(v)f^\infty(\theta,v) \right) = 0, \qquad  v\in V\subseteq \mathbb R. 
\end{equation}
Taking advantage of the condition \eqref{eq:cond_gcons} we can observe that \eqref{eq:FP_general} admits several equivalent formulations. Indeed we have for all $t\ge 0$
\[
\begin{split}
\mathcal B[g](v)f(\theta,v,t) + \partial_v(D(v)f(\theta,v,t)) &= D(v)f(\theta,v,t)\left( \dfrac{\mathcal B[g](v)}{D(v)} + \partial_v \log(D(v)f(\theta,v,t)) \right) \\
& = D(v)f(\theta,v,t)\left(\partial_v \log(D(v)f(\theta,v,t)) - \partial_v \log(D(v)f^\infty(\theta,v)) \right)
\end{split}
\]
In particular from \eqref{eq:cond_gcons} it follows that the Fokker-Planck equation \eqref{eq:FP_general} with constant background can be rewritten in Landau form
\[
\partial_t f(\theta,v,t) = \partial_v \left[ D(v)f(\theta,v,t) \partial_v \log \dfrac{f(\theta,v,t)}{f^\infty(\theta,v)} \right], 
\]
or equivalently in the non-logarithmic Landau form
\[
\partial_t f(\theta,v,t) = \partial_v \left[ D(v)f^\infty(\theta,v) \partial_v \dfrac{f(\theta,v,t)}{f^\infty(\theta,v)} \right].
\]
From these reformulations we can obtain the evolution for $F(\theta,v,t) = \frac{f(\theta,v,t)}{f^\infty(\theta,v)}$ since
\[
\begin{split}
\partial_t f(\theta,v,t) &= f^\infty(\theta,v)\partial_t F(\theta,v,t) \\
& = D(v)f^\infty(\theta,v)\partial_v^2 F(\theta,v,t) + \partial_v(D(v)f^\infty(\theta,v)) \partial_v F(\theta,v,t).
\end{split}\]
Therefore, from \eqref{eq:cond_gcons} we get
\begin{equation}\label{eq:F}
\partial_t F(\theta,v,t) = D(v)\partial_v^2 F(\theta,v,t) - \mathcal B[g](v)\partial_v F(\theta,v,t),
\end{equation}
with no-flux boundary conditions
\[
D(v)f^\infty(\theta,v) \partial_v F(\theta,v,t) \Big|_{v \in \partial V} = 0.
\]

In the following we summarise several findings on the entropy production of Fokker-Planck problems.  Their importance to detected the correct trends to equilibrium is essential in kinetic theory, we point the interested reader to \cite{ToscQA} for more details in the classic Fokker-Planck setting and to \cite{MJT} for results in the nonconstant diffusions setting. 

In the case of interaction with a constant background distribution the following result holds
\begin{thm}
Let the smooth function $\Phi(x)$, $x\in \RR_+$ be convex. Then, if $F(\theta,v,t)$ is the solution to \eqref{eq:F} in $V\subseteq\mathbb R$ and $ F(\theta,v,t)$ is bounded for all $\theta\in I_{\Theta}$ the functional 
\[
\mathcal H(f,f^\infty)(\theta,t) = \int_{V} f^\infty(\theta,v) \Phi(F(\theta,v,t))dv
\]
is monotonically decreasing in time and its evolution is given by 
\[
\dfrac{d}{dt} \mathcal H(f,f^\infty)(\theta,t)= -\mathcal{I}(f,f^\infty)(\theta,t),
\]
where with $\mathcal I$ we denote the nonnegative quantity
\[
\mathcal{I}(f,f^\infty)(\theta,t) = \int_V D(v)f^\infty(\theta,v) \Phi^{\prime\prime}(F(\theta,v,t))\left|\partial_v F(\theta,v,t) \right|^2 dv. 
\]
\end{thm}

\begin{proof}
The proof of this result follows the strategy adopted in \cite{FPTT} for all $\theta\in I_{\Theta}$. 
\end{proof}

Now in the case $\Phi(x) = x\log(x)$ we obtain the relative Shannon entropy $\mathcal H(f,f^\infty)(\theta,t)$ which is a functional depending on the uncertainties of the model. From the above result it follows that this quantity is dissipated with the rate given for all $\theta\in I_{\Theta}$ by 
\[
\begin{split}
\mathcal{I}_\mathcal{H}(f,f^\infty)(\theta,t) & = \int_V D(v) f^\infty(\theta,v) \dfrac{1}{F(\theta,v,t)}\left|\partial_v F(\theta,v,t) \right|^2 dv 
\end{split}
\]
and we have 
\[
\dfrac{d}{dt} \int_V f(\theta,v,t)\log \dfrac{f(\theta,v,t)}{f^\infty(\theta,v)} = -   \int_V D(v) f(\theta,v,t) \left(\dfrac{\partial_v f(\theta,v,t)}{f(\theta,v,t)}- \dfrac{\partial_v f^\infty(\theta,v)}{f^\infty(\theta,v)}\right)^2 dv
\]

In the stochastic Galerkin approximation the relative Shannon entropy for $f^M(\theta,v,t,)$ in \eqref{eq:fM} reads
\[\begin{split}
&\dfrac{d}{dt} \int_V \sum_{k=0}^M \hat f_k(v,t)\Phi_k(\theta)  \log\dfrac{\sum_{k=0}^M \hat f_k(v,t)\Phi_k(\theta)}{\sum_{k=0}^M \hat{f}^\infty_k(v)\Phi_k(\theta)}dv \\
&\qquad = -\int_V D(v) \sum_{k=0}^M\hat f_k(v,t)\Phi_k(\theta)\left( \partial_v \log\dfrac{ \sum_{k=0}^M\hat f_k(v,t)\Phi_k}{ \sum_{k=0}^M \hat{f}^\infty_k(v)\Phi_k(\theta)} \right)^2dv, 
\end{split}\]
from which approximated statistical moments can be obtained by projection in the space defined by the polynomial basis
\[
\begin{split}
\dfrac{d}{dt} \int_{V} \sum_{k=0}^M H_{hk}(v,t) \hat f_k(v,t)dv = - \int_VD(v) \sum_{k=0}^M I_{hk}(v,t) \hat f_k(v,t)dv,
\end{split}
\]
being 
\[\begin{split}
H_{hk} &= \int_{I_{\Theta}} \left(\log f^M(\theta,v,t) - \log f^{M,\infty}(\theta,v)\right)\Phi_h(\theta)d\theta,  \\
I_{hk}  &= \int_{I_{\Theta}} \left(\partial_v\log f^M(\theta,v,t) - \partial_v\log f^{M,\infty}(\theta,v)\right)^2\Phi_h(\theta)d\theta. 
\end{split}\]
We observe that, due to the nonlinearities in the definition of the convex functional $\mathcal H(f,f^\infty)$, a coupled system of differential equations must be solved to estimate the expected trends to equilibrium provided by the relative entropy functional. Nevertheless, at the Galerkin level we have no guarantee that the weighted Fisher information defines a positive quantity for the obtained truncated distribution and, hence, that the entropy monotonically decreases. \\

On the other hand, following the same computations presented in the beginning of this section, the system of $M+1$ projections defined in \eqref{eq:FP_h} can be rewritten for all $h=0,\dots,M$ in the case of fixed background as follows
\be\label{eq:Landauh}
\partial_t \hat f_h(v,t) = \partial_v \left[ D(v)\hat f_h(v,t) \partial_v \log \dfrac{\hat f_h(v,t)}{\hat{f}^\infty_h(v)} \right].
\ee
Therefore, by introducing the ratio $F_h = \frac{\hat f_h(v,t)}{\hat{f}^\infty_h(v)}>0$ we have
\be\label{eq:Fh}
\partial_t F_h = -\mathcal B[g](v) \partial_v F_h(v,t) + D(v) \partial_v^2 F_h(v,t).
\ee
complemented with no-flux boundary conditions. Hence, in analogy with what we discussed above,  the following result holds.
\begin{thm}
Let the smooth function $\Phi(x)$, $x\in \RR_+$ be convex. Then, if $F_h(v,t)$ is the solution to \eqref{eq:Fh} in $V\subseteq\mathbb R$ and $ F_h(v,t)$ is bounded the functional 
\[
{\mathcal{H}}(F_h)(t) = \int_{V} \hat{f}_h^\infty(v) \Phi(F_h(v,t))dv
\]
is monotonically decreasing in time and its evolution is given by 
\[
\dfrac{d}{dt} {\mathcal{H}}(F_h)(t)= -{\mathcal{I}}(F_h)(t),
\]
where with ${\mathcal I}$ we denote the nonnegative quantity
\[
{\mathcal{I}}(F_h)(t) = \int_V D(v)\hat{f}_h^\infty(v) \Phi^{\prime\prime}(F_h(v,t))\left|\partial_v F_h(v,t) \right|^2 dv. 
\]
\end{thm}

Now, in the case of relative Shannon entropy $\Phi(x)=x\log x$ we obtain in each polynomial space
\[
\dfrac{d}{dt} \int_V \hat f_h(v,t) \log \dfrac{\hat f_h(v,t)}{\hat{f}^\infty_h(v)}dv = - \int_V D(v)\hat f_h(v,t)\left(\partial_v \log \dfrac{\hat f_h(v,t)}{\hat{f}^\infty_h(v)} \right)^2dv.
\]
Therefore, each projection of $f(\theta,v,t)$ in the linear space of arbitrary degree $h=0,\dots,M$ converges monotonically in time to its equilibrium $\hat f^{\infty}_h(v)$. In particular this is true for the expected quantities of the problem.



\section{Structure preserving methods}\label{sect:structure}

In this section we introduce the class of so-called structure preserving (SP) numerical methods for the solution of Fokker-Planck equations with nonlocal terms. These methods preserve the fundamental structural properties of the problem like nonnegativity of the solution, entropy dissipation and capture the steady state of each problem with arbitrarily accuracy, see \cite{CCP,DPZ,PZ1,PZ2}. The applications of the SP methods is here particularly appropriate since, thanks to background interactions, the system of $M+1$ equations \eqref{eq:FP_h} is decoupled. 

In the following we summarise the construction ideas at the basis of SP methods in dimension $d=1$, extension to general dimension can be found in \cite{PZ1}. The case of nonlocal Fokker-Planck equations with anisotropic diffusion has been studied in \cite{LZ}.

\subsection{Derivation of the SP method}
For all $h=0,\dots,M$ we may rewrite \eqref{eq:FP_h} in flux form as follows
\[
\partial_t \hat f_h(v,t) = \partial_v \mathcal F[\hat f_h](v,t),
\]
where
\[
\mathcal F[\hat f_h](v,t) = \mathcal C[g](v,t)\hat f_h(v,t) + D(v)\partial_v \hat f_h(v,t),
\]
and $\mathcal C[g](v,t) = \mathcal B[g](v,t) + \partial_v D(v)$. Let us introduce a uniform grid $v_i \in V$, such that $v_{i+1}-v_i= \Delta v>0$ and let $v_{i\pm/2} = v_i \pm \Delta v/2$. We consider the conservative discretization 
\be\label{eq:semidiscrete}
\dfrac{d}{dt} \hat f_{h,i}(t) = \dfrac{\mathcal F_{h,i+1/2}(t)-\mathcal F_{h,i-1/2}(t) }{\Delta v},\qquad t\ge 0
\ee
where $\hat{f}_{h,i}(t)$ is an approximation of $\hat{f}_h(v_i,t)$ and where $\mathcal F_{h,i\pm1/2}$ a numerical flux having the form 
\be\label{eq:Fluxh}
\mathcal F_{h,i+1/2}(t) = \tilde{\mathcal C}[g]_{i+1/2}(t) \tilde{f}_{h,i+1/2}(t) + D_{i+1/2} \dfrac{\hat f_{h,i+1}(t)-\hat f_{h,i}(t)}{\Delta v},
\ee
where 
\[
\tilde f_{h,i+1/2}(t) = (1-\delta_{i+1/2}(t))\hat f_{h,i+1}(t) + \delta_{i+1/2}(t)\hat f_{h,i}(t).
\]
Hence, we aim at finding the weight functions $\delta_{i+1/2}$ and $\tilde{\mathcal C}[g]_{i+1/2}$ such that the scheme produces nonnegative solutions without restrictions on the mesh size $\Delta v$, and is able to capture with arbitrary accuracy the steady state of the \eqref{eq:FP_h} for all $h=0\dots,M$. 

We observe that for a vanishing numerical flux from \eqref{eq:Fluxh} we obtain
\be\label{eq:numratio}
\dfrac{\hat f_{h,i+1}(t)}{\hat f_{h,i}(t)} = \dfrac{-\delta_{i+1/2}(t) \tilde{\mathcal C}_{i+1/2}(t) + \dfrac{D_{i+1/2}}{\Delta v}}{(1-\delta_{i+1/2}(t))\tilde{\mathcal C}_{i+1/2}(t) + \dfrac{D_{i+1/2}}{\Delta v}}.
\ee
At the analytical level we obtained from \eqref{eq:FP_h_stat} in Section \ref{subsect:asymp} that 
\[
D(v)\partial_v \hat f_h(v,t) = -(\mathcal B[g](v,t) + \partial_v D(v))\hat f_h(v,t) ,
\]
which admits the quasi-steady state approximation in the cell $[v_i,v_{i+1}]$ for all $h = 0,\dots,M$
\[
\int_{v_i}^{v_{i+1}} \dfrac{1}{\hat f_h(v,t)}\partial_v \hat f_h(v,t)dv = -\int_{v_i}^{v_{i+1}} \dfrac{1}{D(v)}(\mathcal B[g](v,t) + \partial_v D(v))dv,
\]
that is
\be\label{eq:ratio}
\dfrac{\hat f_h(v_{i+1},t)}{\hat f_h(v_i,t)} = \exp\left\{ - \int_{v_i}^{v_{i+1}}\dfrac{1}{D(v)}(\mathcal B[g](v,t) + \partial_v D(v))dv \right\}.
\ee
Now equating $\hat f_h(v_{i+1},t)/\hat f_h(v_i,t)$ in \eqref{eq:ratio} and $\hat f_{h,i+1}(t)/\hat{f}_{h,i}(t)$ in \eqref{eq:numratio} and setting 
\[
\tilde{\mathcal C}_{i+1/2}(t) = \dfrac{D_{i+1/2}}{\Delta v} \int_{v_i}^{v_{i+1}}\dfrac{1}{D(v)}(\mathcal B[g](v,t) + \partial_v D(v))dv,
\]
we can determine weight functions 
\be\label{eq:weights}
\delta_{i+1/2}(t) = \dfrac{1}{\lambda_{i+1/2}(t)} + \dfrac{1}{1-\exp(\lambda_{i+1/2}(t))}\in(0,1), 
\ee
where
\[
\lambda_{i+1/2}(t) = \int_{v_i}^{v_{i+1}}\dfrac{1}{D(v)}(\mathcal B[g](v,t) + \partial_v D(v))dv = \dfrac{\Delta v \tilde{\mathcal C}_{i+1/2}(t)}{D_{i+1/2}}. 
\]

It is worth pointing out that, by construction, the numerical flux of the SP scheme vanishes when the analytical flux is equal to zero. Furthermore, the long time behavior of \eqref{eq:FP_h_stat} is described with the accuracy with which we evaluate the weights \eqref{eq:weights}. At the numerical level this fact reflects on the considered quadrature methods as we will observe in Section \ref{sect:num}. In the following, we will show that suitable restrictions on the time discretization can be defined to guarantee positivity preservation of the SP scheme, without any restriction on the mesh for $v\in V$. Moreover, we will show that the scheme dissipates the numerical entropy with a rate which is coherent with what we observed in Section \ref{subsect:asymp}. 

\begin{rem}\label{rem2}
The obtained weights do not depend on the degree of the linear space since they are equal for all $h=0,\dots,M$. Furthermore, in the case of interaction with a constant background, i.e. $\mathcal B[g](v,t) = \mathcal B[g](v)$, a stationary state $\hat{f}^{\infty}_h(v)$ is defined for all $h=0\dots,M$, see equation \eqref{eq:FP_h_stat} together with boundary conditions. Hence, thanks to the knowledge of the stationary state in each polynomial space we have
\[
\dfrac{\hat f^\infty_{h,i+1}}{\hat f^\infty_{h,i}} = \exp\left\{ - \int_{v_i}^{v_{i+1}}\dfrac{1}{D(v)}(\mathcal B[g](v) + \partial_v D(v))dv \right\} = \exp\left( -\lambda_{i+1/2}^{\infty}\right).
\]
Which leads to 
\[
\lambda_{i+1/2}^{\infty} = \log\left(\dfrac{\hat f_{h,i}^{\infty}}{\hat f_{h,i+1}^{\infty}} \right),
\]
and 
\[
\delta_{i+1/2}^{\infty} = \dfrac{1}{\log(\hat{f}_{h,i}^\infty)-\log(\hat{f}_{h,i+1}^\infty)} + \dfrac{\hat{f}_{h,i+1}^\infty}{\hat{f}_{h,i+1}^\infty-\hat{f}_{h,i}^\infty}.
\]
We highlight how in this case the dependence on $h=0,\dots,M$ is only apparent since for each times $t\ge 0$ the ratio $\hat f_{h,i+1}/\hat f_{h,i}$ in \eqref{eq:ratio} does not depend on the specific projection thanks to background interactions.
In addition, note that if the steady state is analytically known, for some special form of the operator $\mathcal B[\cdot]$ and of the diffusion function, the SP scheme does not introduce any additional source of errors on the steady state distribution, see \cite{CCP,PZ1}.
\end{rem}

\subsection{Positivity of statistical moments}
In general positivity of the solution, or of its statistical moments, is not achievable once we apply  stochastic Galerkin methods and the solution of the system $f^M(\theta,v,t)$ looses a genuine physical meaning. In this section we provide explicit conditions to preserve nonnegativity of projections  $\hat f_h(v,t)$ and, therefore, of the statistical moments of $f^M(\theta,v,t)$, that have been obtained in Section \ref{sect:2} from direct inspection of the Galerkin projections \eqref{eq:expected}-\eqref{eq:var}. In particular, we will show how in the background interactions case we are able to provide reliable conditions, without restriction on $\Delta v$, for positivity preservation. 

In recent works \cite{CZ,CPZ} a particle scheme has been proposed to enforce positivity of statistical quantities for uncertainty quantification of kinetic models. The core idea of the approach presented in the cited works is to approximate the expected solution of a mean-field type model by its Monte Carlo (MC) formulation  in the phase space, which is then expanded through a SG generalized polynomial chaos (SG-gPC) method. The expected solution is then reconstructed from expected positions and velocities of the microscopic system, which is considered in the gPC setting.  The authors of these works refer to this method as MCgPC. The solution of the MCgPC approach is still spectrally accurate in the random space whereas in the phase space it assumes to accuracy of the Monte Carlo method. The approach presented in the present manuscript for the linear case provide instead global high accuracy. 

Let us introduce the time discretization $t^n = n\Delta t$, $\Delta t>0$ and $n = 0,\dots,T$ and consider the following forward Euler method for all $h=0,\dots,M$
\begin{equation}\label{eq:explicit}
\hat f_{h,i}^{n+1} = \hat f_{h,i}^{n} + \Delta t \dfrac{\mathcal F_{h,i+1/2}^n-\mathcal F_{h,i-1/2}^n}{\Delta v},
\end{equation}
where now $\hat f_{h,i}^n $ is an approximation of $\hat{f}_{h}(v_i,t^n)$ and the flux has the form introduced in \eqref{eq:Fluxh}. We can prove the following result
\begin{thm}\label{thm:explicit_positivity}
Under the time step restriction 
\[
\Delta t\le \dfrac{\Delta v^2}{2\left(M \Delta v  + D \right)}, \qquad M = \max_i | \tilde{\mathcal C}_{i+1/2}^n|, \quad D = \max_i D_{i+1/2}
\]
the explicit scheme \eqref{eq:explicit} preserves nonnegativity, i.e. $\hat f_{h,i}^{n+1}\ge 0$ provided $\hat f_{h,i}^n\ge 0$. 
\end{thm}
\begin{proof}
The proof of this result is analogous for all $h=0,\dots,M$ to the result for explicit scheme obtained  in \cite{PZ1}. 
\end{proof}
We observe that no explicit dependence on the expansion degree $h = 0,\dots,M$ appears in the derived restriction thanks to the background-type interactions. Furthermore, the restriction on $\Delta t$ in Theorem \ref{thm:explicit_positivity} ensures nonnegativity without additional bounds on the spatial grid as for example happen for central type schemes. The derived condition automatically holds for higher order strong stability preserving (SSP) methods like Runge-Kutta and multistep methods since these are convex combinations of the forward Euler integration. {The proved} nonnegativity of the scheme is extended straightforwardly to each SSP-type time integration. 

We highlight how the derived parabolic restriction to enforce nonnegativity of explicit schemes can be quite heavy for practical applications. A convenient strategy to lighten this burden resorts to the technology of semi-implicit methods, see \cite{BFR} for an introduction. Indeed, we can prove nonnegativity of the numerical approximations of solutions $\{\hat f_h\}_{h=0}^M$ by considering the set of modified fluxes
\be\label{eq:flux_SI}
\tilde{ \mathcal F}_{h,i+1/2}^{n+1} = \tilde{\mathcal C}_{i+1/2}^n\left[(1-\delta_{i+1/2}^n)\hat f_{h,i+1}^{n+1} + \delta_{i+1/2}^n\hat f_{h,i}^{n+1} \right] + D_{i+1/2}\dfrac{\hat f_{h,i+1}^{n+1}-\hat f_{h,i}^{n+1}}{\Delta v}.
\ee
The scheme is semi-implicit since we compute the background dependent $\tilde{\mathcal C}_{i+1/2}(t)$ and weight functions $\delta_{i+1/2}(t)$ at time $t^n$. As a consequence, it is easily seen how in the case of a fixed background the scheme is coherent with a fully implicit method. 

The following result holds
\begin{prop}
Let us consider a semi-implicit method for all $h=0,\dots,M$
\[
\hat f_{h,i}^{n+1} = \hat f_{h,i}  + \Delta t \dfrac{\tilde{\mathcal F}_{i+1/2}^{n+1}-\tilde{\mathcal F}_{i-1/2}^{n+1}}{\Delta v}, 
\]
with fluxes defined in \eqref{eq:flux_SI}. Under the time step restriction 
\[
\Delta t< \dfrac{\Delta v}{2M},\qquad M = \max_i |\tilde{\mathcal C}_{i+1/2}^n|,
\]
the semi-implicit scheme preserves nonneagivity, i.e. $\hat f_{h,i}^{n+1}\ge 0$ if $\hat f_{h,i}^n\ge 0$ for all $i=1,\dots,N$ and $h=0,\dots,M$.
\end{prop}
\begin{proof}
The proof of this result is analogous for all $h=0,\dots,M$ to the result for semi-implicit scheme obtained  in \cite{PZ1}. 
\end{proof}

Extensions to higher order semi-implicit schemes have been obtained in \cite{BFR}. 


\subsection{Entropy dissipation}

{In Section \ref{subsect:asymp} we have observe how an entropy can be defined for the introduced class of problems. In details, a Shannon entropy operator has been defined and we proved that it is dissipated in time, leading to interpret the steady state as the zero entropy state of the system.  At the numerical level, a remarkable property of a scheme relies on the correct description of the trends toward equilibrium. In the following we will show how the introduced methods are entropic. }

We concentrate on the case of fixed background. In Section \ref{subsect:asymp} we have seen how the Fokker-Planck problems of interest can be rewritten in Landau form \eqref{eq:Landauh}. In particular, it can be proven how the numerical flux for this reformulation is given by the following equivalent form
\be\label{eq:Fluxh_re}
\mathcal F_{h,i+1/2}(t) = \dfrac{D_{i+1/2}}{\Delta v} \bar{f}_{h,i+1/2}^\infty\left( \dfrac{\hat f_{h,i+1}(t)}{\hat f_{h,i+1}^\infty} -\dfrac{\hat f_{h,i}(t)}{\hat f_{h,i}^\infty} \right),
\ee
with
\[
 \bar{f}_{h,i+1/2}^\infty = \dfrac{\hat f_{h,i+1}^\infty \hat f_{h,i}^\infty}{\hat f_{h,i+1}^\infty-\hat f_{h,i}^\infty} \log \left(\dfrac{\hat f_{h,i+1}^\infty}{\hat f_{h,i}^\infty} \right),
\]
since for all $h=0,\dots,M$ {we are looking at the constant background case} we have $\lambda_{i+1/2}^\infty = \log \hat f_{h,i}^\infty - \log f_{h,i+1}^\infty$ and the weight functions are rewritten as 
\[
\delta_{i+1/2}^\infty = \dfrac{1}{\log \hat f_{h,i}^\infty - \log f_{h,i+1}^\infty} + \dfrac{\hat f_{h,i+1}^\infty}{ \hat f_{h,i+1}^\infty -  f_{h,i}^\infty},
\]
see Remark \ref{rem2}. We can prove the following result
\begin{thm}
Let us consider the conservative discretization \eqref{eq:semidiscrete} for all $t\ge 0$ and $h=0\dots,M$. The numerical flux \eqref{eq:Fluxh} satisfies the discrete entropy dissipation 
\[
\dfrac{d}{dt} \mathcal H_{\Delta v} (\hat{f}_h,\hat{f}^\infty_h)(t)= -\mathcal I_{\Delta}(\hat{f}_h,\hat{f}^\infty_h)(t),
\]
where
\[
\mathcal H_{\Delta v} (\hat f_h,\hat{f}^\infty_h)(t) = \Delta v \sum_{i=0}^N \hat f_{h,i}(t) \log\left( \dfrac{\hat f_{h,i}(t)}{\hat f_{h,i}^\infty}\right),
\]
and 
\[
\mathcal I_{\Delta v}(\hat f_{h},\hat{f}^\infty_{h})(t)= \sum_{i=0}^N \left[ \log\left( \dfrac{\hat f_{h,i+1}(t)}{\hat{f}_{h,i+1}^\infty} \right) - \left(\dfrac{\hat f_{h,i}(t)}{\hat{f}_{h,i}^\infty}\right) \right] \cdot \left( \dfrac{\hat f_{h,i+1}(t)}{\hat{f}_{h,i+1}^\infty} - \dfrac{\hat f_{h,i}(t)}{\hat{f}_{h,i}^\infty} \right) \bar{f}_{h,i+1/2}^\infty D_{i+1/2} \ge 0.
\]
\end{thm}
\begin{proof}
From the definition of relative entropy for all $h=0,\dots,M$ we have 
\[
\begin{split}
\dfrac{d}{dt} \mathcal H(\hat f_{h},\hat f_{h}^\infty)(t) &= \Delta v \sum_{i=0}^N \dfrac{d\hat f_h(t)}{dt} \left( \log \left(\dfrac{\hat f_{h,i}(t)}{\hat f_{h,i}^\infty} \right)+1\right) \\
&= \sum_{i=0}^N \left( \log \left(\dfrac{\hat f_{h,i}(t)}{\hat f_{h,i}^\infty} \right)+1\right)\left(\mathcal F_{h,i+1/2}(t)- \mathcal F_{h,i-1/2}(t)\right)
\end{split}
\]
After summation by parts we have
\[
\dfrac{d}{dt} \mathcal H(\hat f_{h},\hat f_{h}^\infty) (t) = -\sum_{i=0}^N \left[ \log \left(\dfrac{\hat f_{h,i+1}(t)}{\hat f_{h,i+1}^\infty} \right) -  \log \left(\dfrac{\hat f_{h,i}(t)}{\hat f_{h,i}^\infty} \right) \right] \mathcal F_{h,i+1/2}(t),
\]
and from the reformulation of the flux in \eqref{eq:Fluxh_re} we may conclude since $(x-y)\log(x/y)\ge 0$ for all $x,y\ge0$.
\end{proof}

\section{Numerical tests}\label{sect:num}
{In the present section, we present several tests for Fokker-Planck equations with background interactions and uncertain initial distribution. In particular, we numerically show the performance of the discussedstructure preserving stochastic Galerkin (SP-SG) method. Both the scenarios with fixed and evolving backgrounds will be considered. }
As discussed in Section \ref{sect:structure} the essential aspect for the accurate computation of the large time distribution of the problem \eqref{eq:FP_general} lies in the numerical approximation of the integral 
\[
\lambda_{i+1/2} = \int_{v_i}^{v_{i+1}} \dfrac{1}{D(v)}(\mathcal B[g](v,t)+D^\prime(v))dv,
\]
which defines the quasi-stationary states of each projection. In general a high order quadrature method is needed. {In the following numerical examples we will consider open Newton-Cotes quadrature methods up to the $6th$ order and the Gauss-Legendre quadrature, see e.g. \cite{DB}}. Through the text we will refer to these methods as $SP_k$, $k=2,4,6,G$, where the index $k$ indicates the order of the adopted quadrature method with $G$ referring to the Gauss-Legendre case. To highlight the advantages of this approach, a nonconstant diffusion function is considered for bounded domains. In all the tests we considered suitable restrictions on the time discretization to guarantee positivity of the expected solution of the problems both in the explicit and semi-implicit integration. Extension to the multidimensional case is considered at the end of this section.

\subsection{Test 1: Stationary background distribution}

Let us consider the evolution of a distribution function $f(\theta,v,t)$ in the presence of uncertainty that follows \eqref{eq:FP_general}, with $v\in[-1,1]$, and interacting with a given background distribution $g(v,t) = g(v)$ for all $t\ge 0$ of the form
\be\label{eq:g_fixed}
g(v) = \beta \exp\left\{ - \dfrac{(w-u_g)^2}{2\sigma_g^2} \right\}, \qquad u_g \in(-1,1), \sigma_g^2=0.01,
\ee
with $\beta>0$ a constant such that $\int_{-1}^1 g(v)dv = 1$. We consider in this test a nonconstant diffusion $D(v)=\frac{\sigma^2}{2} (1-v^2)^2$ with given $\sigma^2$ that will be specified later on. Furthermore, the nonlocal operator in \eqref{eq:B} is defined in terms of the interaction function
\be\label{eq:BCinter}
P(v,v_*) = \chi(|v-v_*|\le \Delta),
\ee
where $\Delta >0$ is a constant measuring the maximal distance under which interactions may occur. The introduced function $P(\cdot,\cdot)$ is usually defined as bounded confidence function. This model has been proposed in the literature to describe the evolution of the distribution of agents having opinion $v$ at time $t\ge 0$, see \cite{PT,T}. In particular, the presence of background interactions is generally considered to take into account the influence of external actors in opinion dynamics like the case of media \cite{BMS} or the action of possible control strategies \cite{APTZ}. Extensions to the case of uncertain interactions have been proposed in \cite{TZ}. 

\begin{figure}
\centering
\includegraphics[scale=0.45]{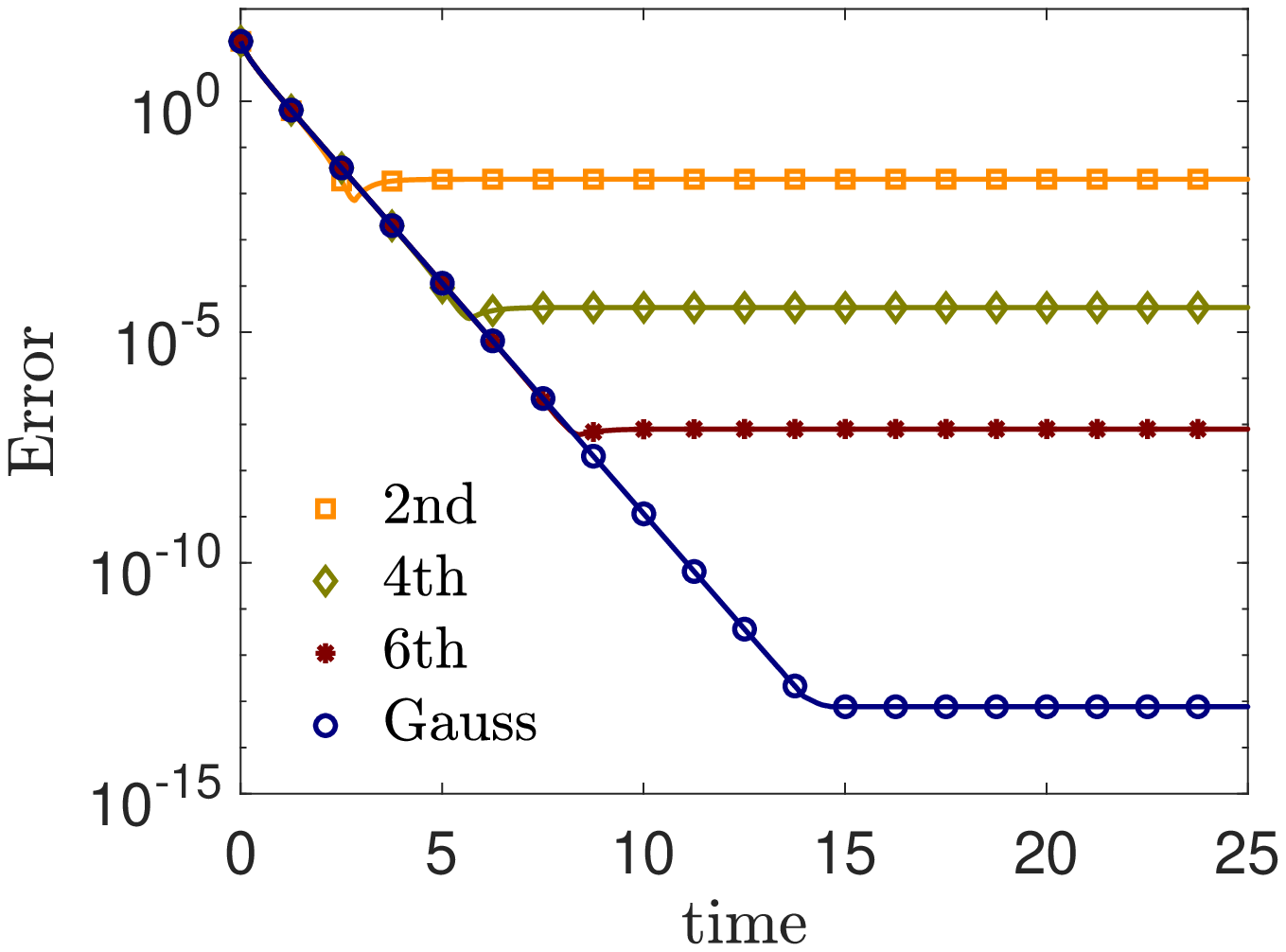}
\includegraphics[scale=0.45]{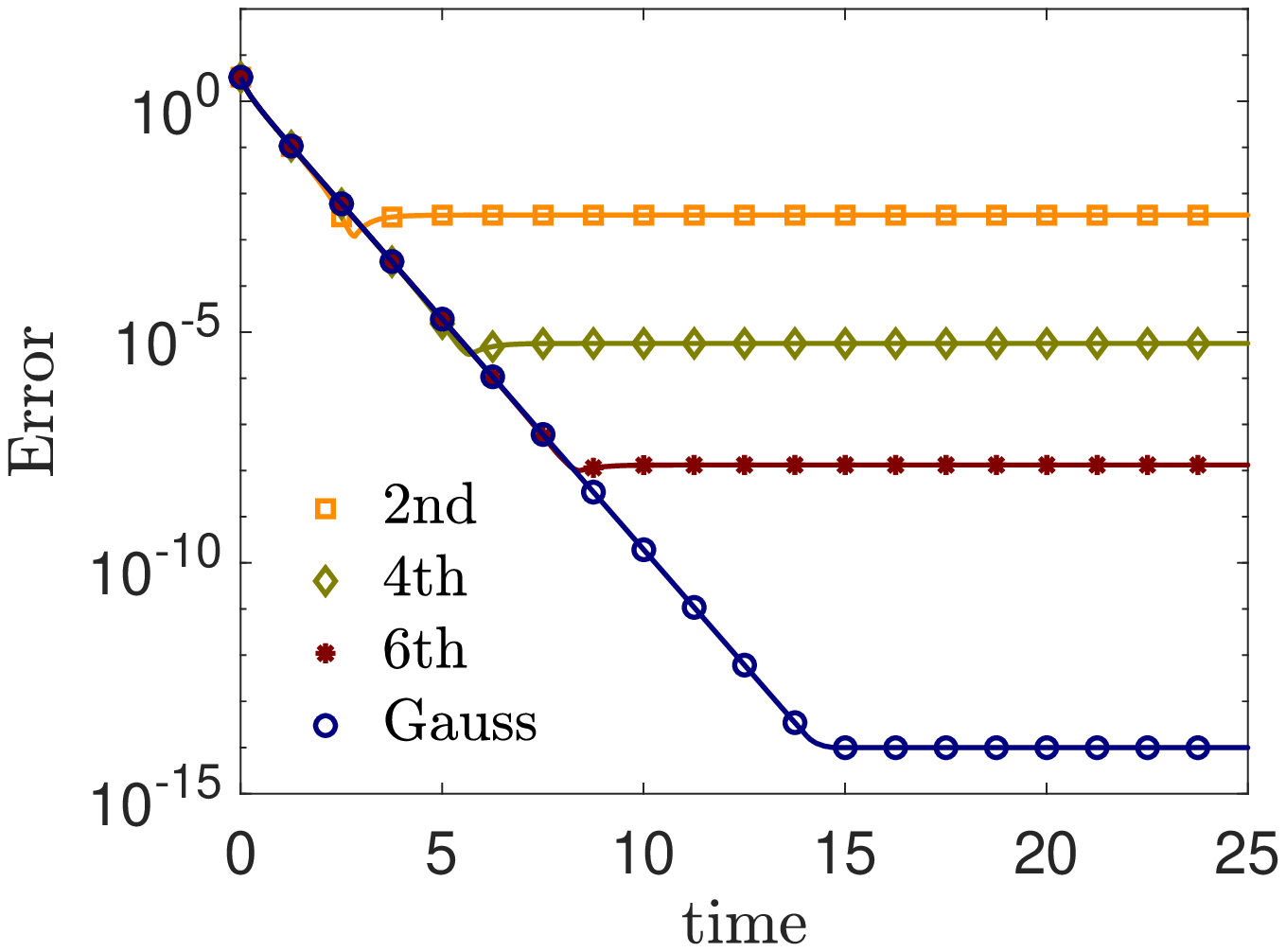} \\
\includegraphics[scale=0.45]{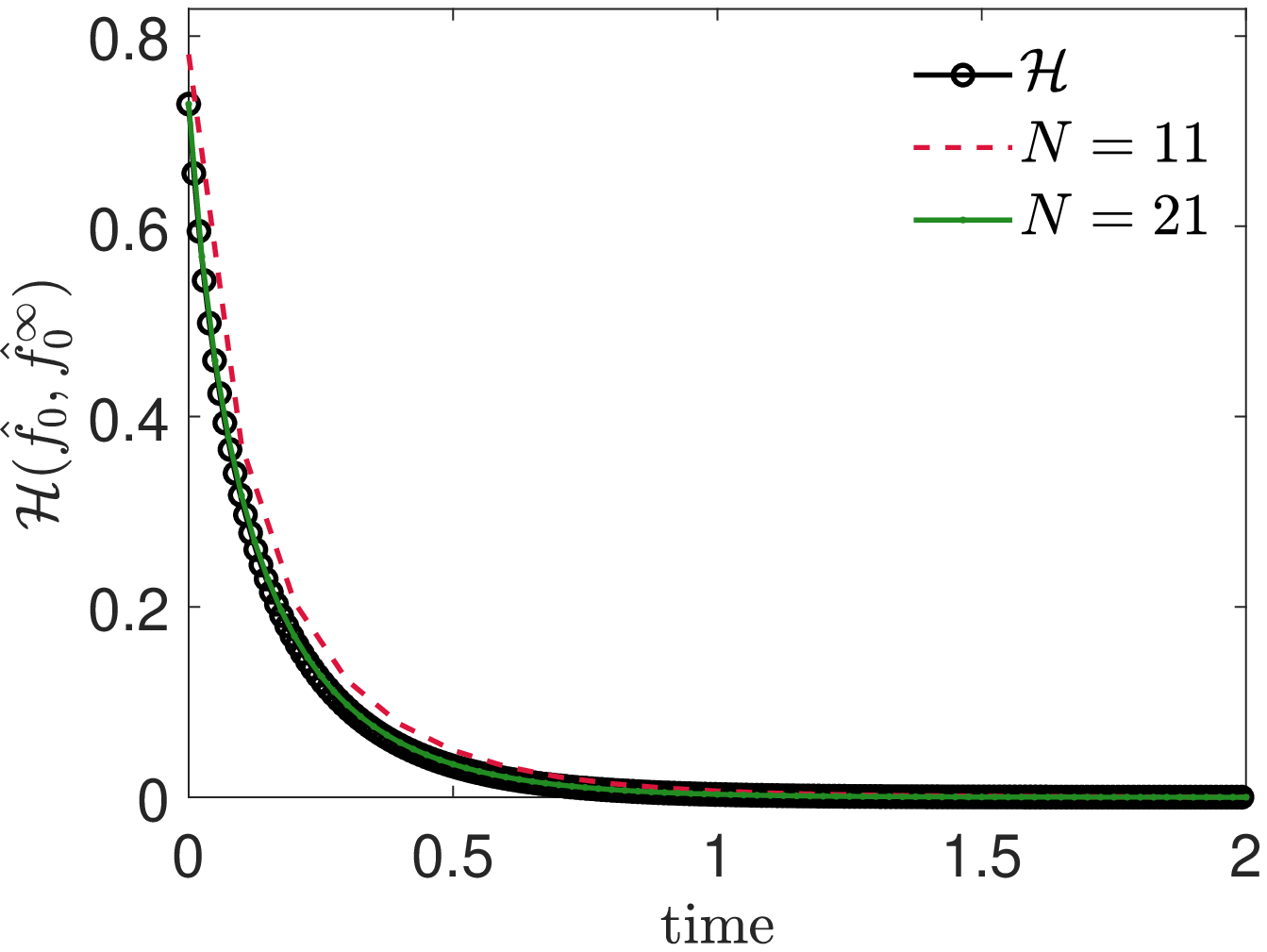}
\includegraphics[scale=0.45]{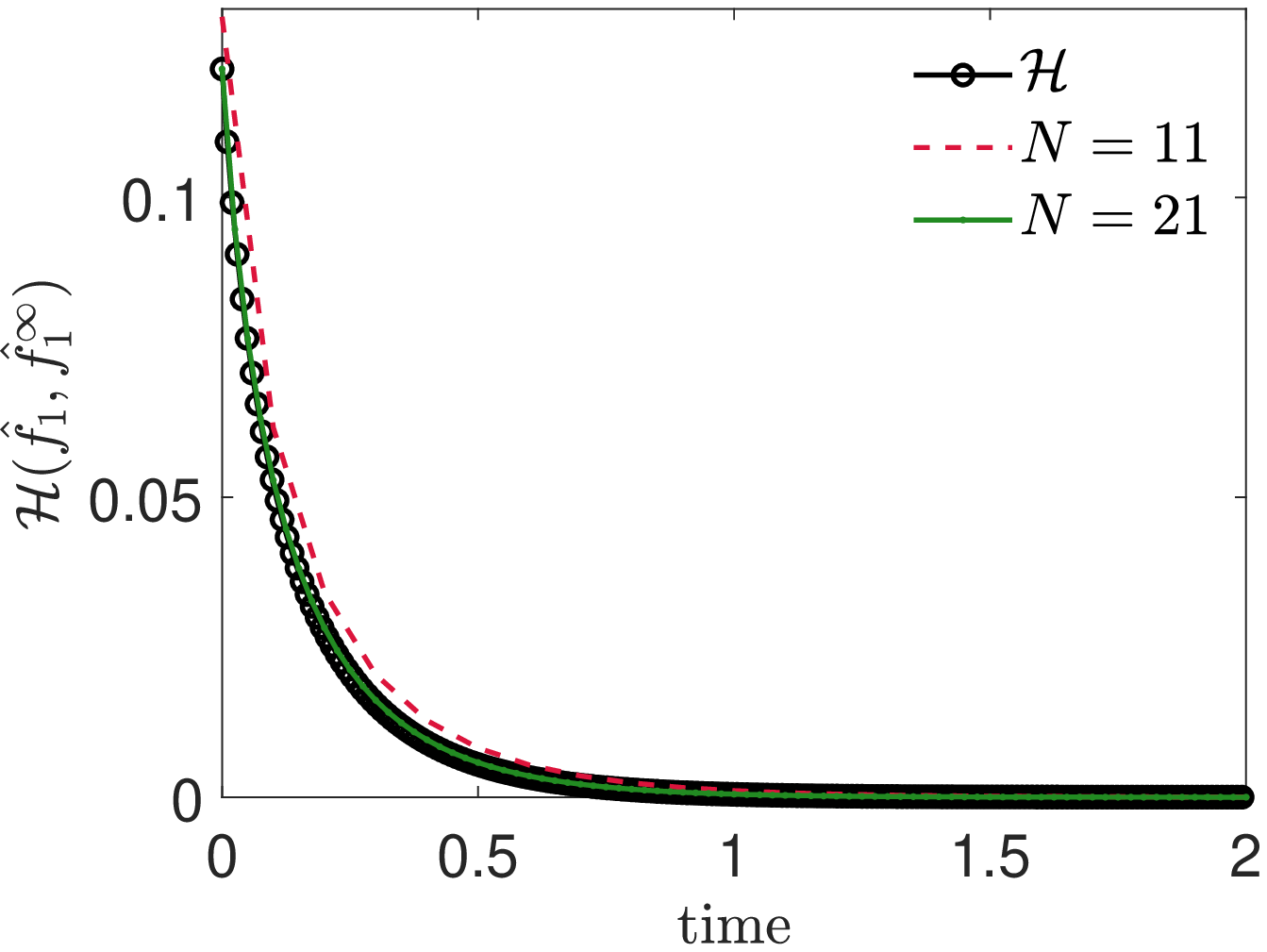}
\caption{Example 1. \textbf{Top row}: evolution of the $L^1$ relative error with respect to the stationary solution \eqref{eq:stat_h_num} for the $SP_k$ scheme with different quadrature methods. We considered the initial uncertain distribution $f(\theta,v,0)$ in \eqref{eq:f0_test1} with $\bar u = 0.25$, $\rho(\theta) = 1 + 0.5\theta$, $\theta\sim \mathcal U([-1,1])$, and $\sigma^2 = 2\cdot10^{-1}$. For all $h$ the solution has been computed for $N=41$ gridpoints over the time interval $[0,25]$, $\Delta t = \Delta v^2/(2\sigma^2)$. \textbf{Bottom row}: dissipation of the numerical entropies $\mathcal H(\hat f_0,\hat f_0^{\infty})$, $\mathcal H(\hat f_1,\hat f_1^{\infty})$ for the $SP_k$ scheme with Gaussian quadrature for two coarse grids with $N=11$ and $N=21$ gridpoints.  }
\label{fig:1}
\end{figure}

In this first test we consider as initial distribution 
\be\label{eq:f0_test1}
f(\theta,v,0) = C(\theta) \left[ \exp\left(-\dfrac{(v-u_1(\theta))^2}{2\sigma^2_0} \right) + \exp\left(-\dfrac{(v-u_2(\theta))^2}{2\sigma^2_0} \right)\right],
\ee
with $C(\theta)$ such that $\int_V f(\theta,v,0) = \rho(\theta)>0$ for all $\theta\in I_{\Theta}$ and $u_i(\theta)$, $i=1,2$ given by 
\[
u_1(\theta) = \bar u + \kappa\,\theta,\qquad u_2(\theta) = - \bar u + \kappa\,\theta,
\]
being $\theta\sim \mathcal U([-1,1])$. In the case $\Delta = 2$ it follows that $P\equiv 1$ and we can compute the explicit stationary distribution
\[
f^\infty(\theta,v) = \dfrac{C(\theta)}{(1-v^2)^2} \left(\dfrac{1+v}{1-v} \right)^{u_g/(2\sigma^2)} \exp\left\{-\dfrac{1-u_g v}{\sigma^2(1-v^2)}\right\}.
\]

The stochastic Galerkin decomposition of the resulting problem can be performed by considering a Legendre polynomial basis $\{\Phi_h \}_{h=0}^M$ being $\Psi(\theta) =\frac{1}{2} \chi(\theta\in[-1,1])$. The resulting system of equations have the form \eqref{eq:FP_h} whose asymptotic solution for all $h=0,\dots,M$ reads
\be\label{eq:stat_h_num}
\hat{f}_h^{\infty}(v) = \dfrac{C_h}{(1-v^2)^2} \left(\dfrac{1+v}{1-v} \right)^{u_g/(2\sigma^2)} \exp\left\{-\dfrac{1-u_g v}{\sigma_f^2(1-v^2)}\right\}.
\ee
being $C_h = \frac{1}{2} \int_{I_\Theta}C(\theta)\Phi_h(\theta)d\theta$. 
In Figure \ref{fig:1} we present the evolution of the $L^1$ relative error computed with respect to the exact stationary state for the $SP_k$, $k=2,4,6,G$, schemes for various quadrature methods. To exemplify the advantages, we consider two projections $h=0$ (left) and $h=1$ (right). For each $SP_k$ we considered $N= 41$ gridpoints for the discretization of the state variable.

 We can observe how we achieve different accuracy in terms of the steady states of the problem in relation to the considered quadrature rules for both $h=0,1$. Further, with low order quadrature we approach to the numerical steady state of the method faster than with high order rules. We observe that with a Gauss-Legendre method we essentially reach machine precision in finite time for each projection. In the same figure, we show the dissipation of the relative entropy functional $\mathcal H(\hat f_h,\hat f_h^\infty)$ discussed in Section \ref{subsect:asymp} with $h=0,1$ obtained with the structure preserving method. We present the case of two coarse grids obtained with $N=11$, $N=21$ gridpoints compare with the exact dissipation of the relative entropy. 

\begin{figure}
\centering
\includegraphics[scale=0.45]{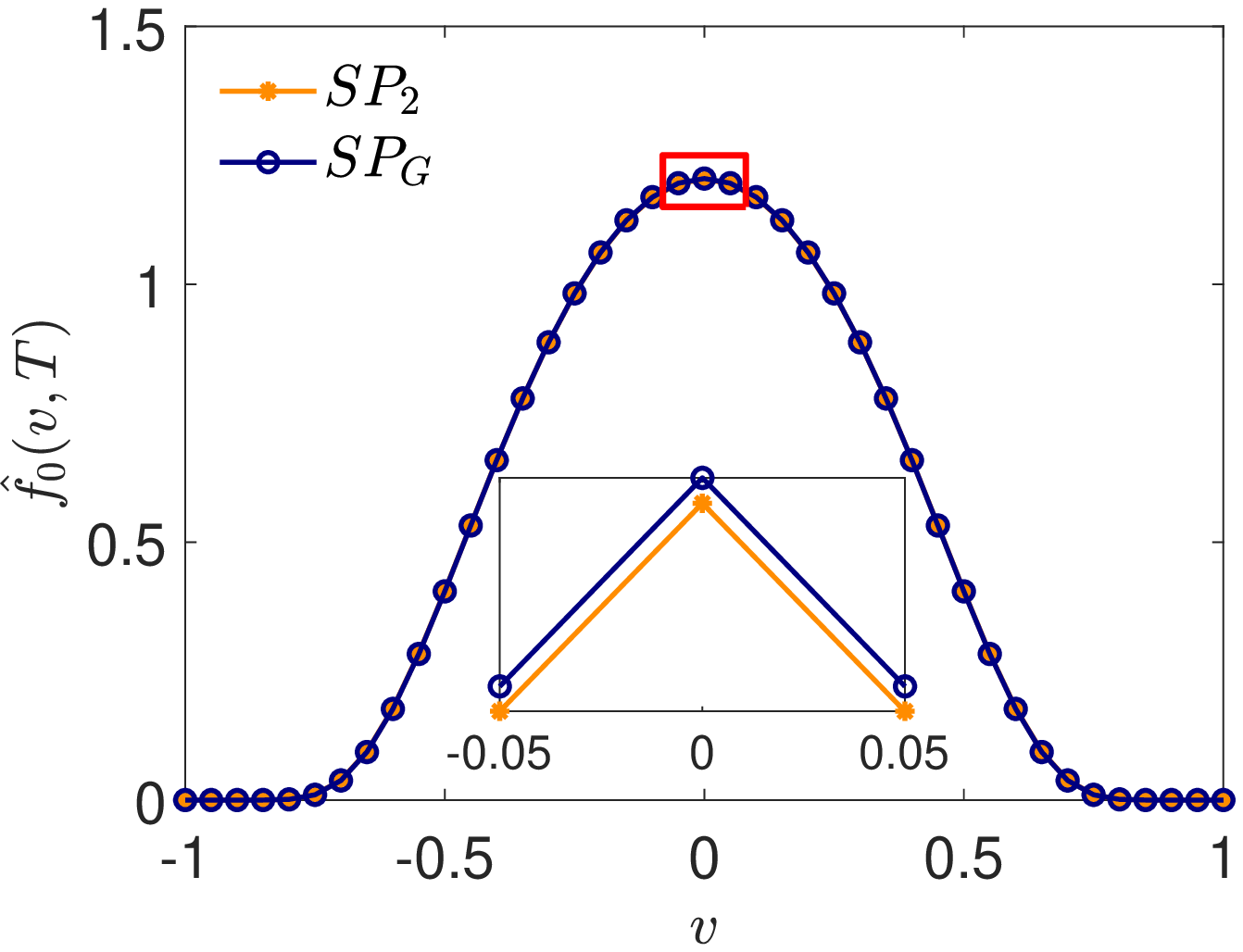}
\includegraphics[scale=0.45]{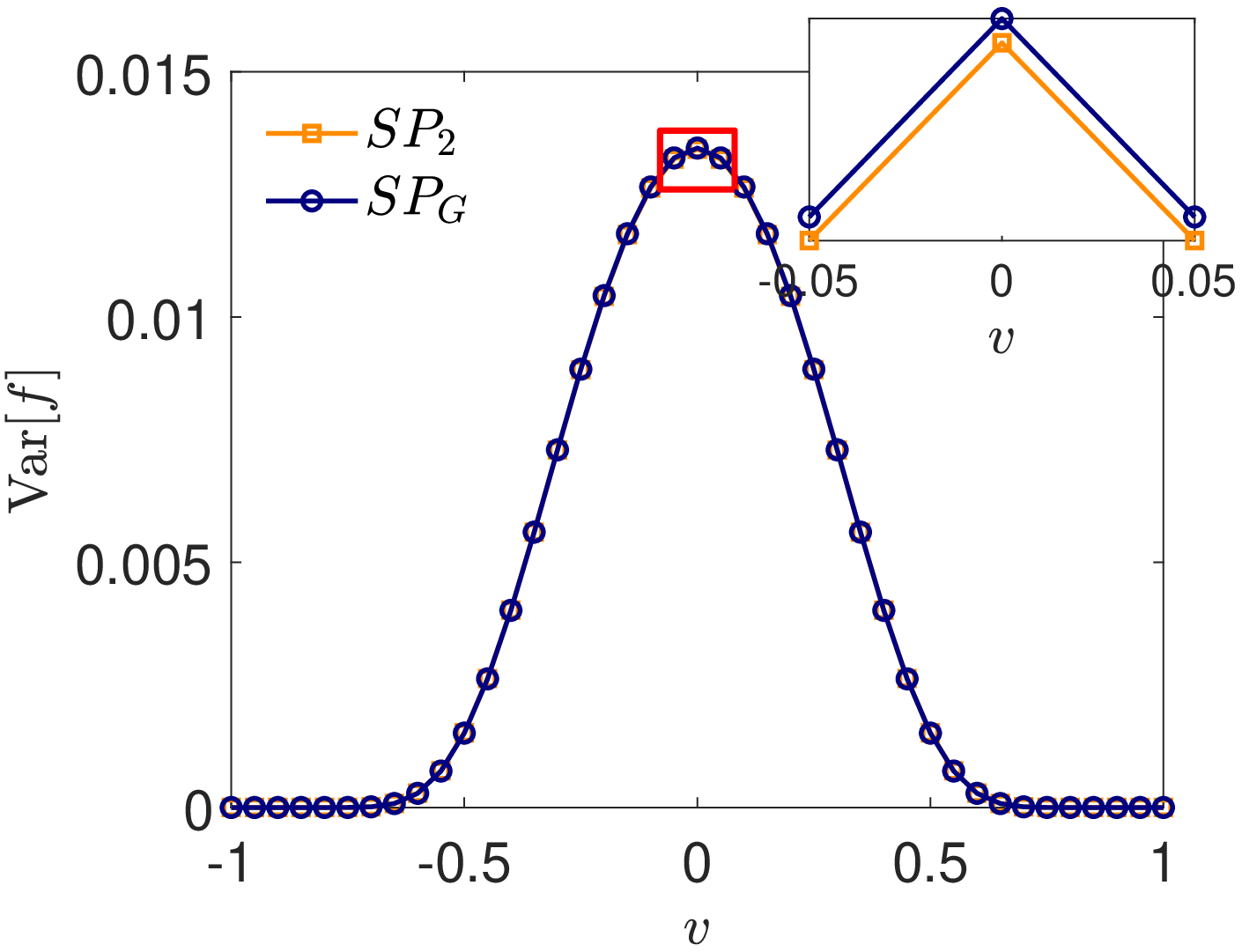}
\caption{Example 1. Large time behavior of expectation (left) and variance (right) of $f(\theta,v,t)$ obtained with $SP_k$ schemes and $k=2,G$ and an uncertain initial distribution of the form \eqref{eq:f0_test1}. We can observe how the high accuracy of the proposed scheme reflects in an arbitrary accurate numerical description of the large time statistical moments of the solution of the problem. $t\in[0,T]$ with $T = 15$ and $N=41$, $\Delta t = \Delta v^2/(2\sigma^2)$. }
\label{fig:fstat}
\end{figure}

The high accuracy of the scheme in the description of the large time behavior, in each polynomial space, reflects in a high accuracy in the approximation of statistical moments of the solution of the problem, see Figure \ref{fig:fstat}. Here, we considered the schemes $SP_2$ and $SP_G$, that is the structure preserving schemes with approximation of $\lambda_{i+1/2}$ with a $2$nd order and Gauss-Legendre method respectively. We highlight how the approximated expectation $\hat f_0$ is positive thanks to the properties of the scheme.

\begin{table}
\begin{center}
\begin{tabular}{ c || c c c c }
\hline
$\mathbb E[f]$ &  \multicolumn{4}{c} {$SP_k$} \\ \hline

             Time         &       2   &  4  &  6  & G    \\ 
             \hline

\multirow{1}{*} {1}  & 2.0785 &  1.9989 & 2.0025 & 2.0026 \\
                             \hline
\multirow{1}{*}{5}  & 1.9949 & 4.2572 & 2.2868 & 2.3361  \\
                             \hline
\multirow{1}{*}{10} & 1.9953  & 3.9141 & 6.4698 & 7.3367  \\
                             \hline
                             \hline
$\textrm{Var}(f)$ &  \multicolumn{4}{c} {$SP_k$} \\ \hline
		Time         &     2   &  4  &  6  & G    \\ 
             \hline
\multirow{1}{*} {1}& 2.0870 &  2.0001 & 2.0030 & 2.0031 \\
                             \hline
\multirow{1}{*}{5}  & 1.9978 & 4.4192 & 2.2398 & 2.2789  \\
                             \hline
\multirow{1}{*}{10} & 1.9982  & 3.9309 & 6.6929 & 7.3405  \\
                             \hline
\end{tabular}
\caption{{Example 1}. Estimation of the order of convergence toward the reference stationary state for SP--CC scheme with  RK4 method. Rates have been computed using $N=21,41,81$, $\sigma^2/2=0.1$, $\Delta t = \Delta w^2/\sigma^2$. }
\label{tab:opinion}
\end{center}
\end{table}

In Table \ref{tab:opinion} we estimate the order of convergence of the SP-SG method in terms of accuracy of the expected quantities $\mathbb E[f]$, and $\textrm{Var}(f)$ in their stochastic Galerkin approximation \eqref{eq:expected}-\eqref{eq:var}. It is easily observed how for the approximation of the variance it is required the solutions of the whole set of projections $h=0,\dots,M$. In the present test we considered $M = 10$. Furthermore, we used $N = 21,41,81$ and the order of convergence of the explicit structure preserving schemes is measured as $\log_2 \frac{e_1(t)}{e_2(t)}$, where $e_1(t)$ is the relative error at time $t\ge 0$ of the expected solution and its variance computed with $N =21$ gridpoints with respect to that computed with $N=41$ gridpoints and, likewise, $e_2(t)$ is the relative error at time $t\ge 0$ computed with $N=41$ with respect to that computed with $N= 81$ gridpoints. The time integration has been performed with RK4 at each time step chosen in such a way that the restriction for positivity of the scheme in Theorem \ref{thm:explicit_positivity} is satisfied, i.e. $\Delta t = O(\Delta v^2)$. We can observe how the $SP_k$ schemes are second order accurate in the transient regimes and assume the order of the quadrature method near the expected steady state and its related variance. \\

\begin{figure}
\centering
\subfigure[$t=0$]{
\includegraphics[scale=0.45]{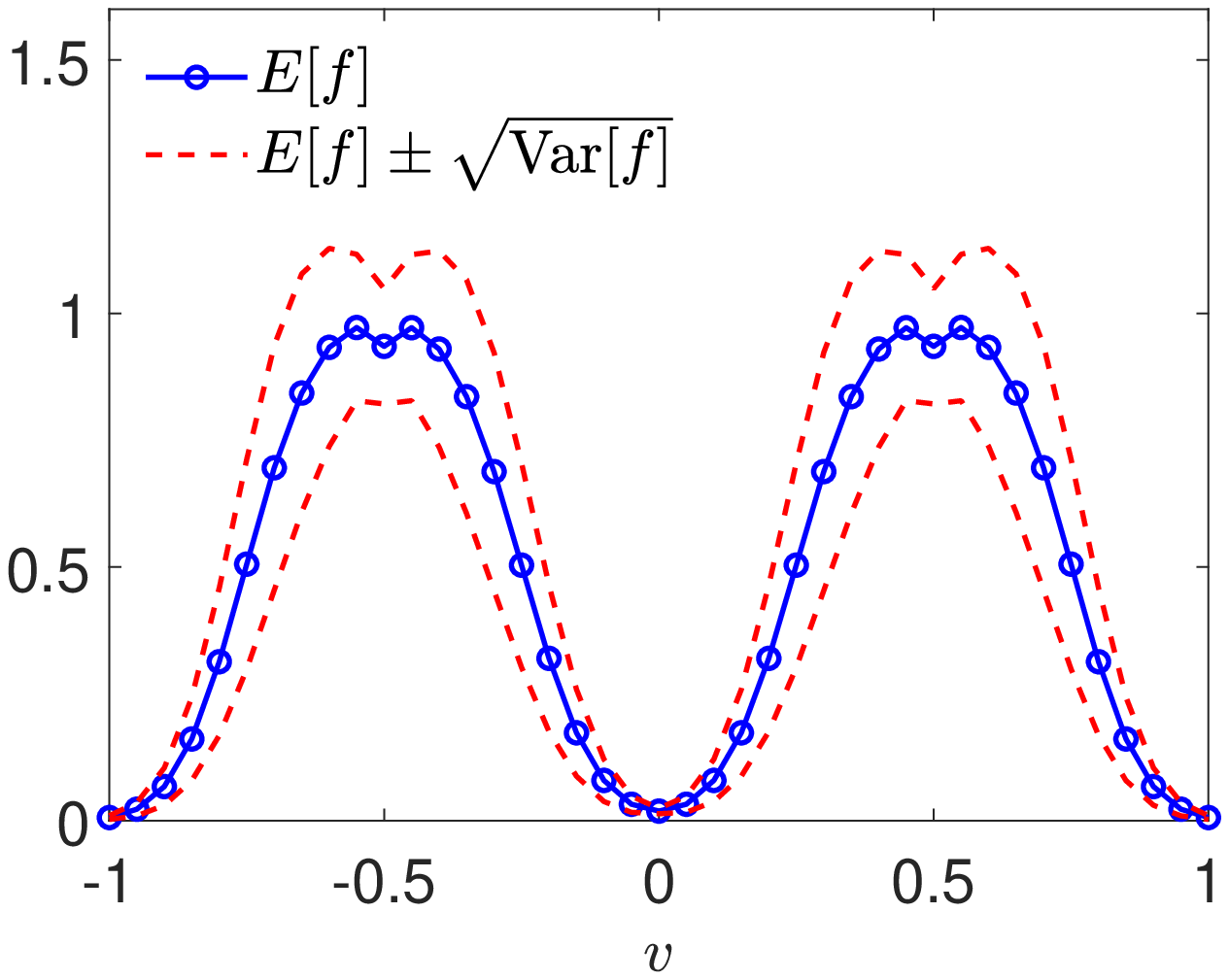}}
\subfigure[$t=10$]{
\includegraphics[scale=0.45]{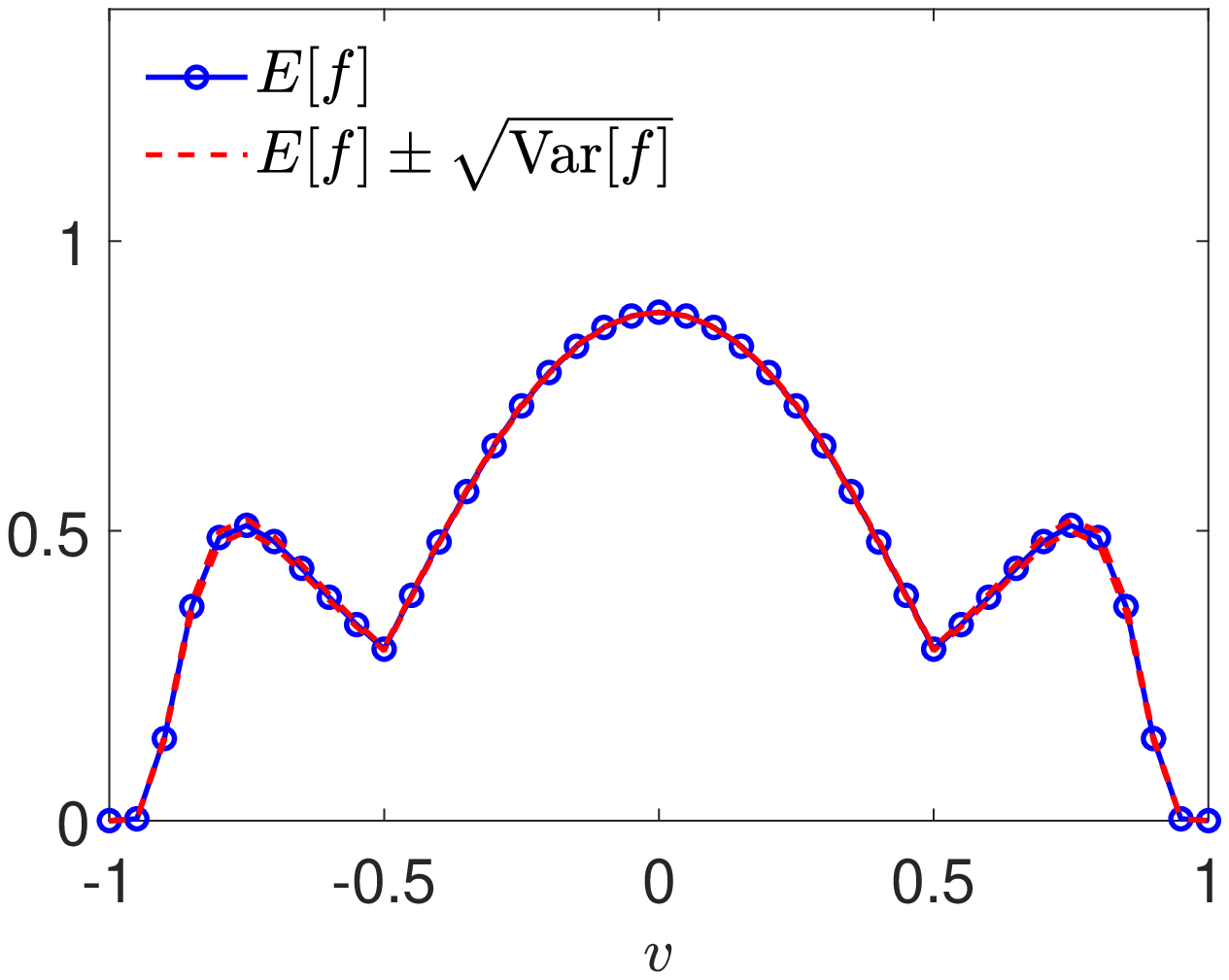}}\\
\subfigure[$\mathbb{E}(f)$]{
\includegraphics[scale=0.45]{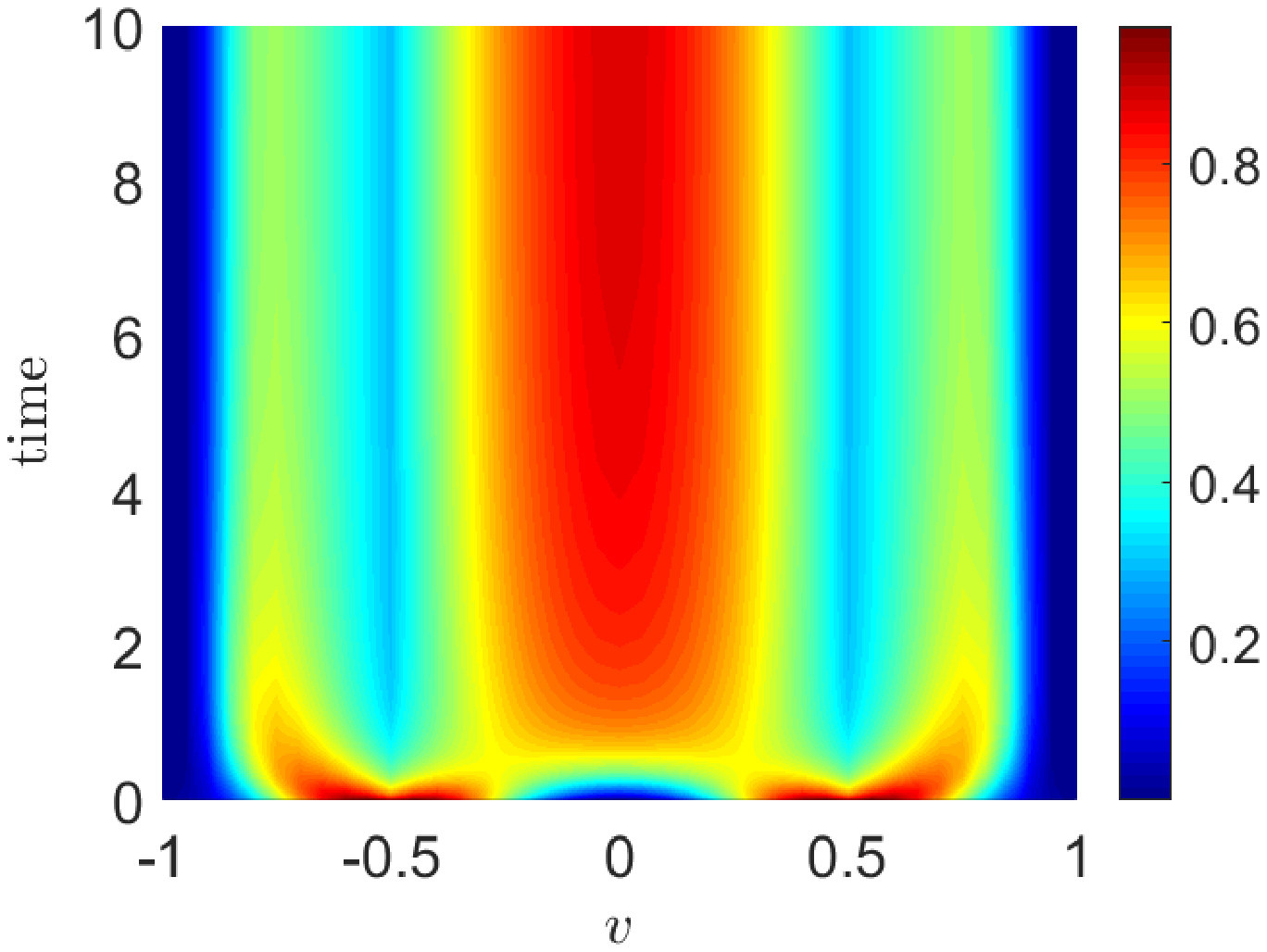}}
\subfigure[$\textrm{Var}(f)$]{
\includegraphics[scale=0.45]{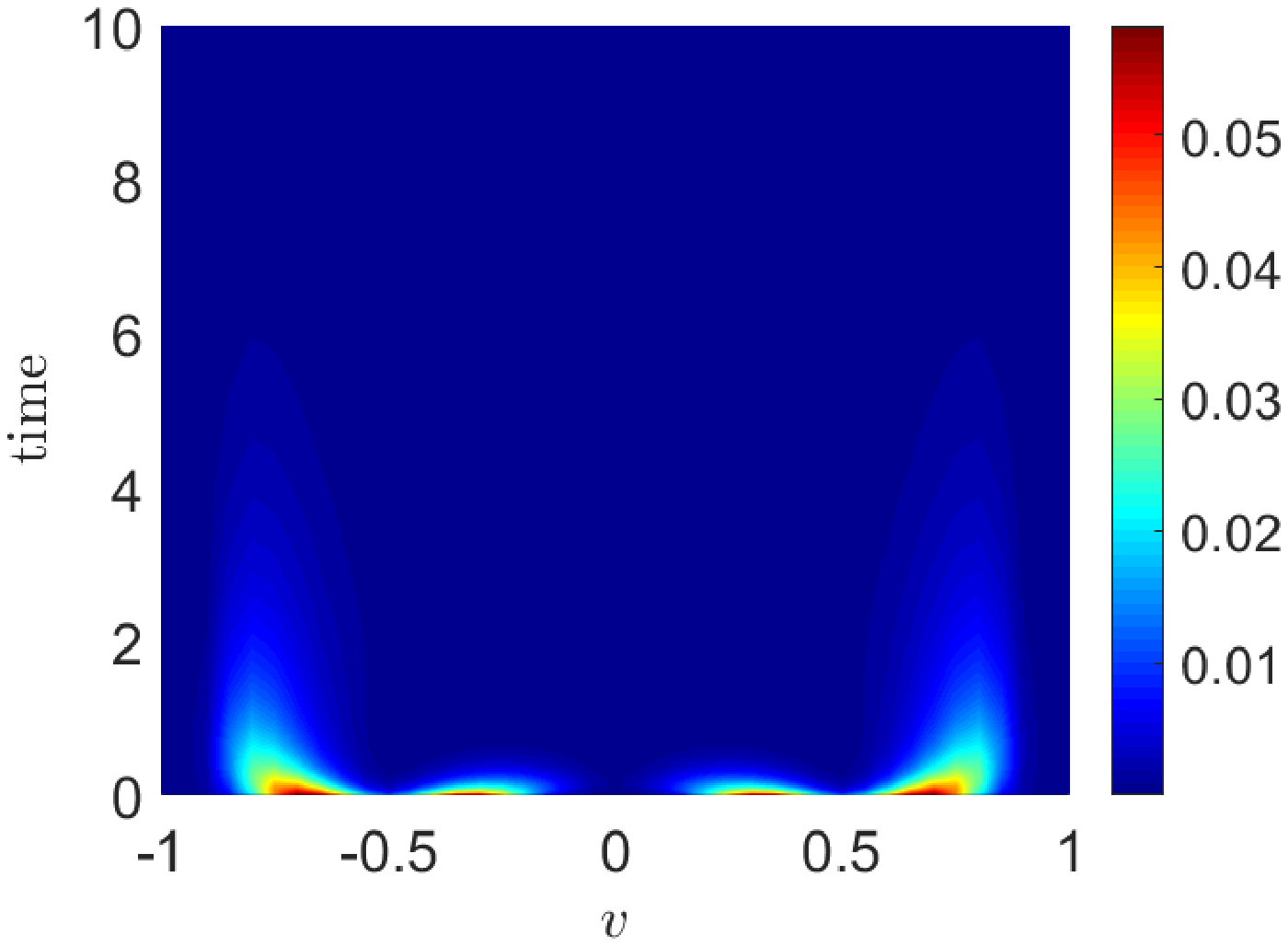}}
\caption{Example 1. \textbf{Top row}: initial distribution and solution at time $T=10$ in the case of bounded confidence interactions and $\Delta=1$ obtained with $SP_G$, $N=41$ gridpoints and $M=5$ projections. \textbf{Bottom row}: evolution of the expected solution (left) and its variance (right) in the interval $[0,10]$. }
\label{fig:BC}
\end{figure}

{More generally, $\Delta <2$ and therefore $P(v,v_*)\ne 1$, so that we have no analytical insight on the large time solution $\hat{f}^\infty_h(v)$ in each polynomial space. }In Figure \ref{fig:BC} we consider the case of bounded confidence type interactions \eqref{eq:BCinter} with $\Delta = 1.0$ and a fixed background distribution $g(v)$ of the form 
\[
g(v) = \beta\left( \exp\left\{ - \dfrac{(v-u_g)^2}{2\sigma_g^2} \right\} + \exp\left\{ - \dfrac{(v+u_g)^2}{2\sigma_g^2} \right\}\right), 
\]
with $u_g = \frac{1}{2} $ and $\sigma_g^2 = 10^{-2}$. We considered the uncertain initial density \eqref{eq:f0_test1} with deterministic initial mass $\rho(\theta) = 1$ and uncertainty in $u_1(\theta)$, $u_2(\theta)$ so that
\[ 
u_1= \dfrac{1}{2} + \dfrac{1}{4}\, \theta, \qquad u_2= -\dfrac{1}{2} + \dfrac{1}{4}\, \theta,
\]
with $\theta\sim \mathcal U([-1,1])$.  The integral $\mathcal B[g](v)$ has been evaluated through a trapezoidal rule. As observed in Section \ref{subsect:asymp} the large time solution for all $h=0,\dots,M$ does not depend on the uncertainties of the initial distribution. Indeed, the variance annihilates as we can observe in \ref{fig:BC}(d) and the asymptotic state coincides with $\mathbb E[f]$.

\subsection{Example 2: Evolving background distribution}

In this section we test the performance of the introduced structure preserving stochastic Galerkin scheme in the case of an evolving background distribution. To exemplify a dynamic background distribution we consider the solution of a linear advection equation 
\be\label{eq:dyn_g}
\partial_t g(v,t) + \alpha \partial_v g(v,t) = 0, \qquad \alpha>0,
\ee
which is coupled to the original stochastic Fokker-Planck equation in \eqref{eq:FP_general} through the operator $\mathcal B[g](v,t)$. The initial background is considered of the form \eqref{eq:g_fixed}, with $u_g = -\frac{1}{2}$, we consider periodic boundary conditions for \eqref{eq:dyn_g} and $\alpha = 0.05$. The advection equation is solved numerically with a Lax-Wendroff scheme for each time $t\ge0$. In the following we consider as uncertain initial distribution \eqref{eq:f0_test1} with $\bar u = 0.5$, $\kappa = 0.25$, and the mass is $\rho(\theta) = 1+\frac{1}{2}\theta$ with uniform perturbation $\theta\sim\mathcal U([-1,1])$.

\begin{figure}
\centering
\subfigure[$g(v,t)$]{
\includegraphics[scale=0.45]{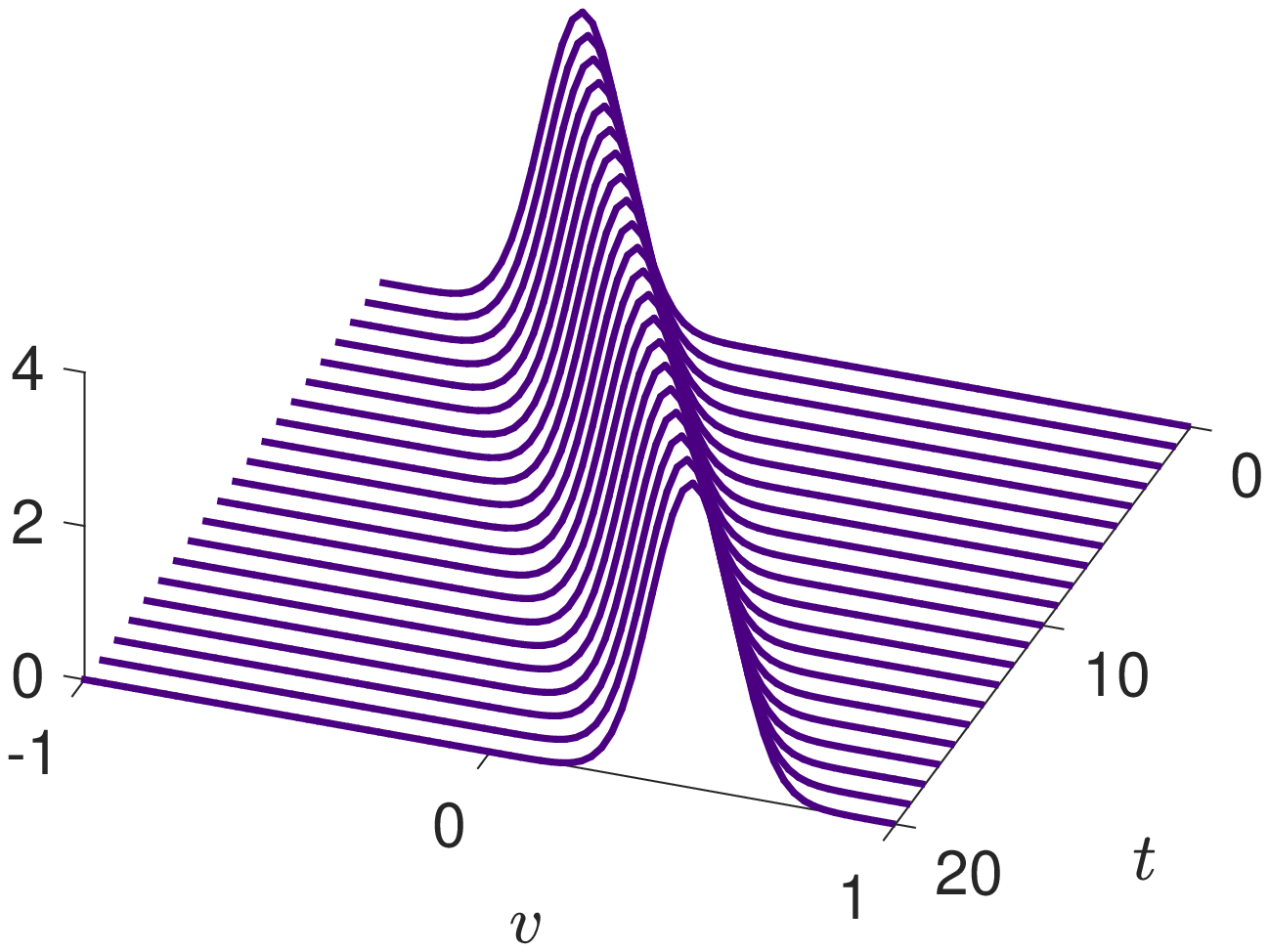}}
\subfigure[$t=20$]{
\includegraphics[scale=0.45]{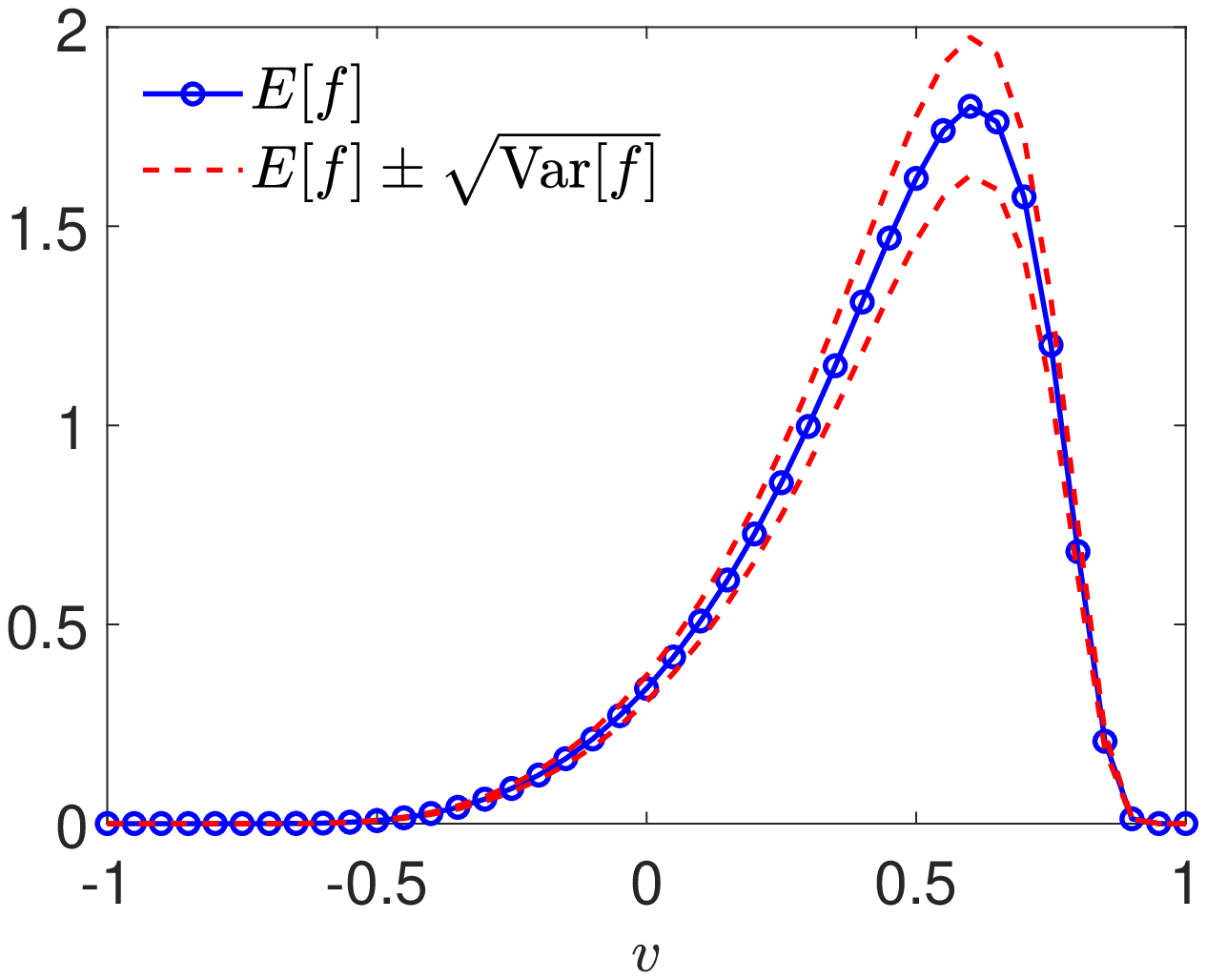} }\\
\subfigure[$\mathbb E(f)$]{
\includegraphics[scale=0.45]{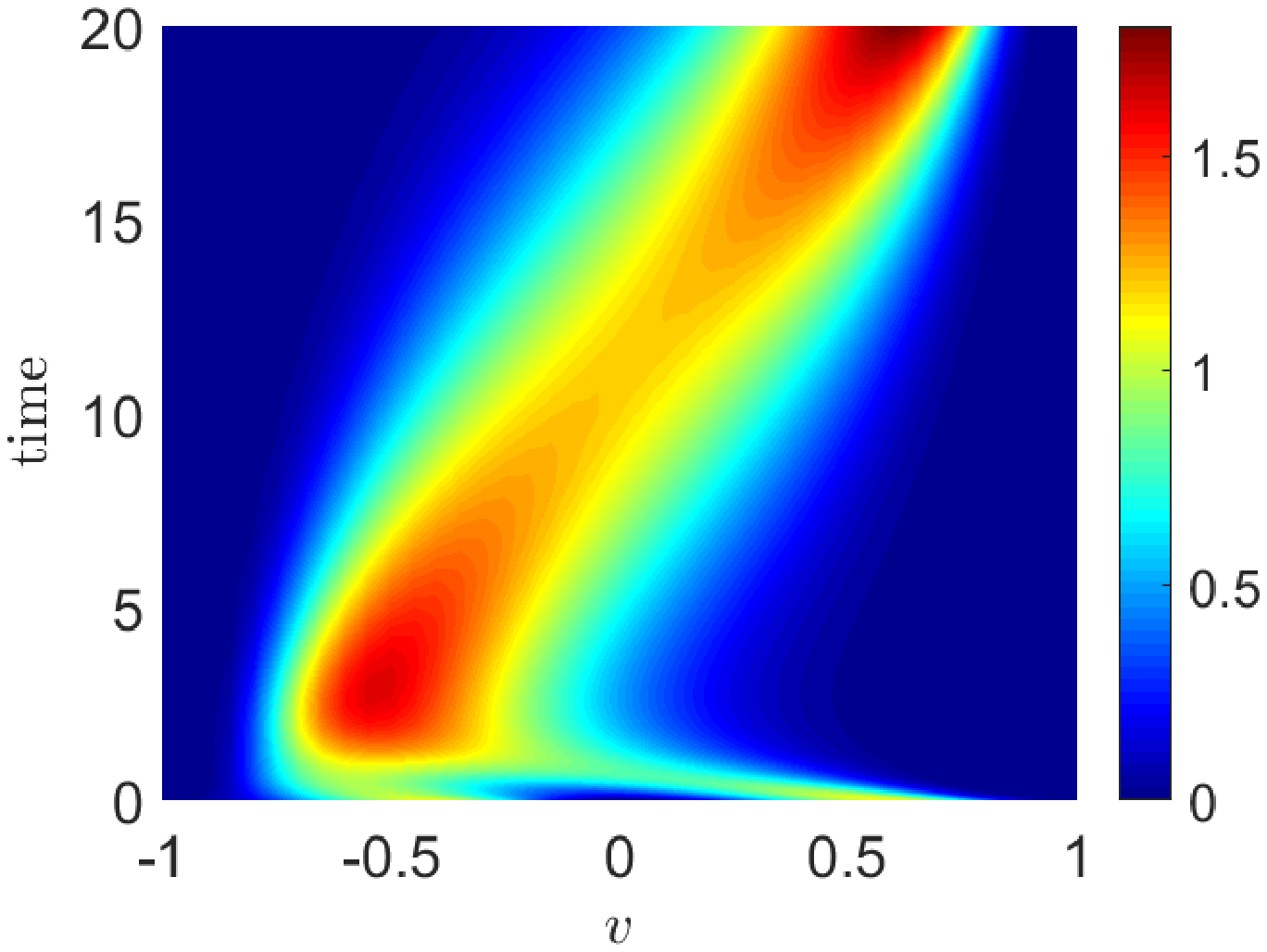}}
\subfigure[$\textrm{Var}(f)$]{
\includegraphics[scale=0.45]{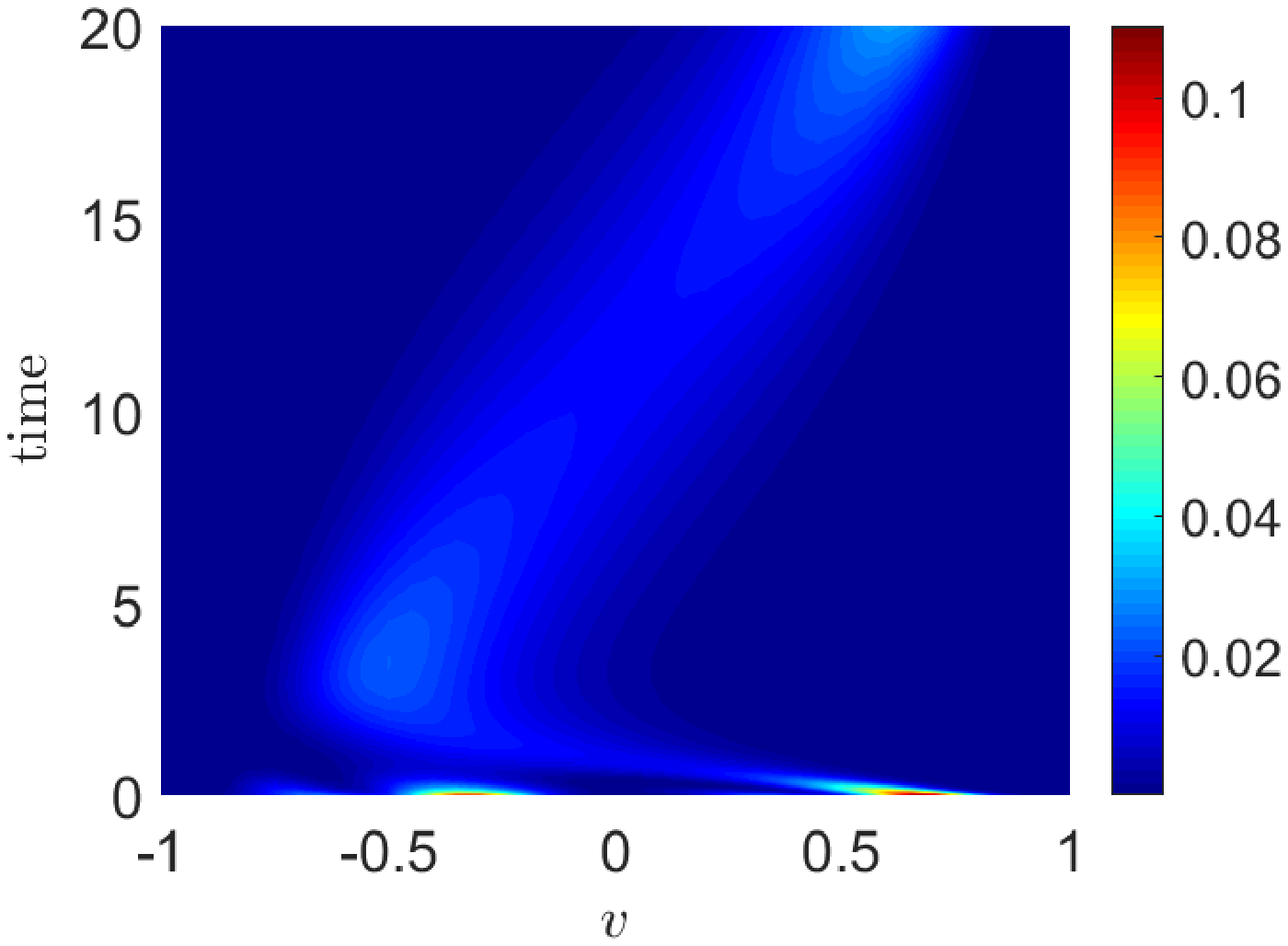}}
\caption{Example 2. \textbf{Top row}: (left): evolution of the background distribution according to the linear advection equation \eqref{eq:dyn_g} with $\alpha = 0.05$, (right) expected solution of the \eqref{eq:FP_general} and bounded confidence interactions with $\Delta = 1.0$ obtained with the stochastic Galerkin $SP_G$ scheme and semi-implicit time integration for $h= 0,\dots,5$, in red we represent the estimated confidence bands. \textbf{Bottom row}: evolution over the time interval $[0,20]$ of the expected solution and of its variance. We considered $N = 41$ and $\Delta t = \textrm{CFL} \Delta v$, $\textrm{CFL} = 0.5$ so that the solution of the scheme advection equation is stable. }
\label{fig:ex2}
\end{figure} 

 In Table \ref{tab:opinion_moving} we estimate the order of convergence of the stochastic Galerkin structure preserving scheme in terms of the two statistical quantities $\mathbb E(f)$, $\textrm{Var}(f)$ in the case of a background distribution evolving as \eqref{eq:dyn_g}.  The numerical integration is a second-order semi-implicit method, see \cite{BFR,PZ1}. The approximation of the variance we considered $M=5$ projections. We can observe that the dynamic background distribution prevents the formation of steady state solution in the original Fokker-Planck problem. Indeed, for each $SP_k$, $k=2,4,6,G$ the scheme  initially increases its order according to the quadrature method and for large times it is reduced to the initial second-order. 
 
 In Figure \ref{fig:ex2} we can observe the behavior of \eqref{eq:FP_general} in the bounded domain $v\in[-1,1]$ and interacting through a bounded confidence type $P(v,v_*)$ with $\Delta = 1$, where the evolving background follows the advection \eqref{eq:dyn_g}.  

\begin{table}
\begin{center}
\begin{tabular}{ c || c c c c }
\hline
$\mathbb E[f]$ &   \multicolumn{4}{c} {$SP_k$} \\ \hline

             Time         &      2   &  4  &  6  & G    \\ 
             \hline

\multirow{1}{*} {1}        
						& 1.8976 &  2.0834 & 2.0972 & 2.0975 \\
                             \hline
\multirow{1}{*}{5}        
						& 2.4162 & 4.7225 & 4.7940 & 4.7955  \\
                             \hline
\multirow{1}{*}{10}        
						 & 2.6446  & 4.5139 & 4.5082 & 4.5082  \\
			          \hline
\multirow{1}{*}{20}        
						 & 2.0685  & 2.5829 & 2.6271 & 2.6273   \\
                             \hline
                             \hline
$\textrm{Var}(f)$ &  \multicolumn{4}{c} {$SP_k$} \\ \hline
		Time         &          2   &  4  &  6  & G    \\ 
             \hline
\multirow{1}{*} {1}         
						& 1.8834 &  2.1088 & 1.8565 & 1.8568 \\
                             \hline
\multirow{1}{*}{5}         
						& 2.4162 & 4.3191 & 4.3937 & 4.3953  \\
                             \hline
\multirow{1}{*}{10}        
						 & 2.5793  & 4.5869 & 4.5809 & 4.5809  \\
                             \hline
 \multirow{1}{*}{20}        
						 & 2.2616  & 2.4154 & 2.4678 & 2.4679  \\
                             \hline
\end{tabular}
\caption{{Example 2}. Estimation of the order of convergence of the scheme in the case of dynamic background $g(v,t)$ for second order semi-implicit method. The evolution of the background distribution follows an advection equation with $\alpha = 0.05$. The rates have been computed using $N=21,41,81$, $\sigma^2/2=0.1$, $\Delta t = \textrm{CFL}\Delta v$ with $\textrm{CFL} = 0.5$.  }
\label{tab:opinion_moving}
\end{center}
\end{table}

\subsection{Example 3: 2D model of swarming}

We consider a kinetic swarming model with self-propulsion and diffusion with uncertain initial distribution. In a deterministic framework we refer to \cite{BCCD} where a similar model has been studied in the nonlocal setting. {Here the authors proved a sharp phase transition between self-propulsion forces and diffusion, that discriminate the minimal amount of noise needed to obtain symmetric distribution with zero mean.  }
The study of possible uncertain quantities in the dynamics is here of paramount importance, since coefficients like the noise intensity and the self-propulsion strength, are commonly based on field observations and empirical evidence. We refer to \cite{CZ} for a more detailed analysis of the influence of uncertain quantities in problems with phase transition. In the following, we concentrate on the case of uncertain initial distribution. 

We consider a model for the evolution of the density of individuals $f = f(\theta,v,t)$ having velocity $v\in \RR^{2}$ at time $t\ge 0$ and uncertain initial condition $f(\theta,v,0)$ having mass $\rho(\theta)$. In details the model reads
\be\label{eq:swarming}
\partial_t f(\theta,v,t) = \nabla_v \cdot \left[\alpha v(|v|^2-1)f(\theta,v,t) + (v-u_g)f(\theta,v,t) + D\nabla_v f(\theta,v,t) \right],
\ee
being $\alpha>0$ the self-propulsion strength and $D>0$ a constant noise intensity. At a difference with the original nonlinear case here the agents interact with a background distribution $g(v)$ through its mean velocity $u_g = \int_V v g(v)dv$. It may be shown that a free energy functional is defined which dissipates along solutions. Further, stationary solutions have the form 
\[
f^\infty(\theta,v) = C(\theta) \exp\left\{ -\dfrac{1}{D}\left[ \alpha \dfrac{|v |^4}{4}+(1-\alpha)\dfrac{|v|^2}{2}-u_g \cdot v\right] \right\},
\]
with $C(\theta)>0$ a normalization constant. In particular, we focus on the 2D case and we consider the fixed background distribution
\be\label{eq:back_swarming}
\begin{split}
g(v) = \dfrac{1}{2\pi \sigma_{x}\sigma_y} \exp\left\{-\dfrac{1}{2}\left[\dfrac{(v_x-\mu_x)^2}{\sigma_x^2} + \dfrac{(v_y-\mu_y)^2}{\sigma_y^2} \right] \right\}, \qquad v = (v_x,v_y).
\end{split}
\ee

\begin{figure}
\centering
\subfigure[$D=0.8$, $\mu_x=\mu_y = 0$]{
\includegraphics[scale=0.5]{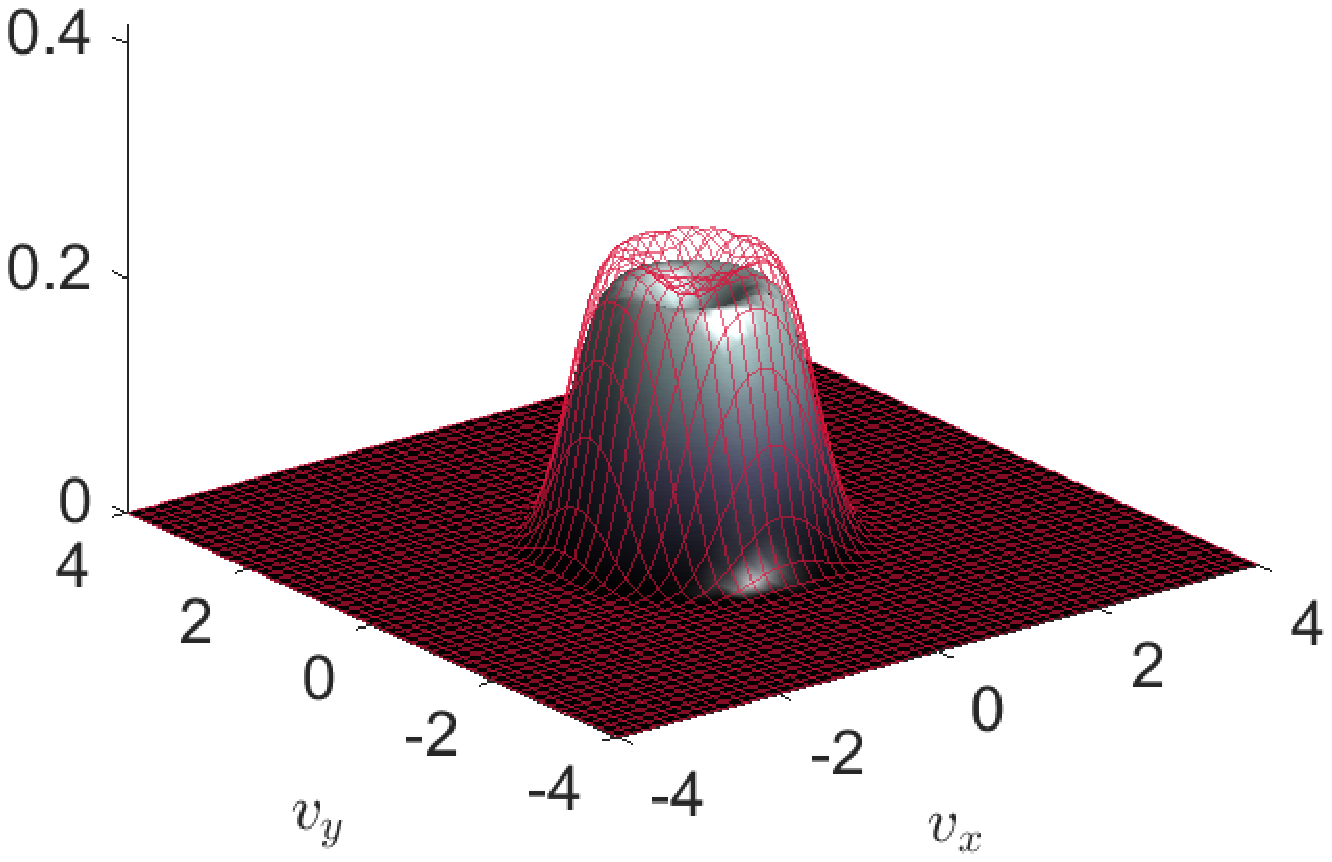}}
\hspace{-2em}
\subfigure[$D=0.2$, $\mu_x = \mu_y = 0$]{
\includegraphics[scale=0.5]{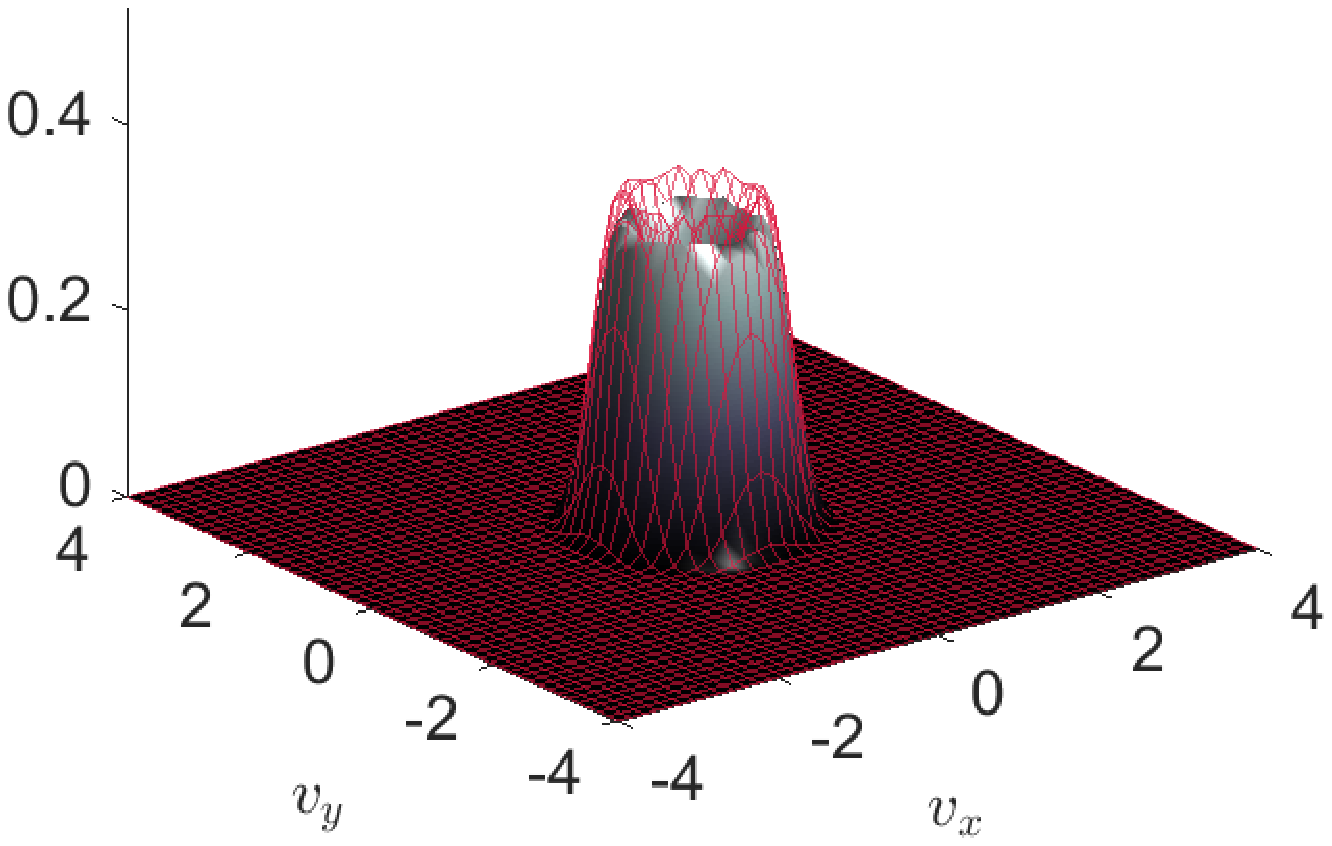}} \\
\subfigure[$D=0.8$, $\mu_x=\mu_y = 1$]{
\includegraphics[scale=0.5]{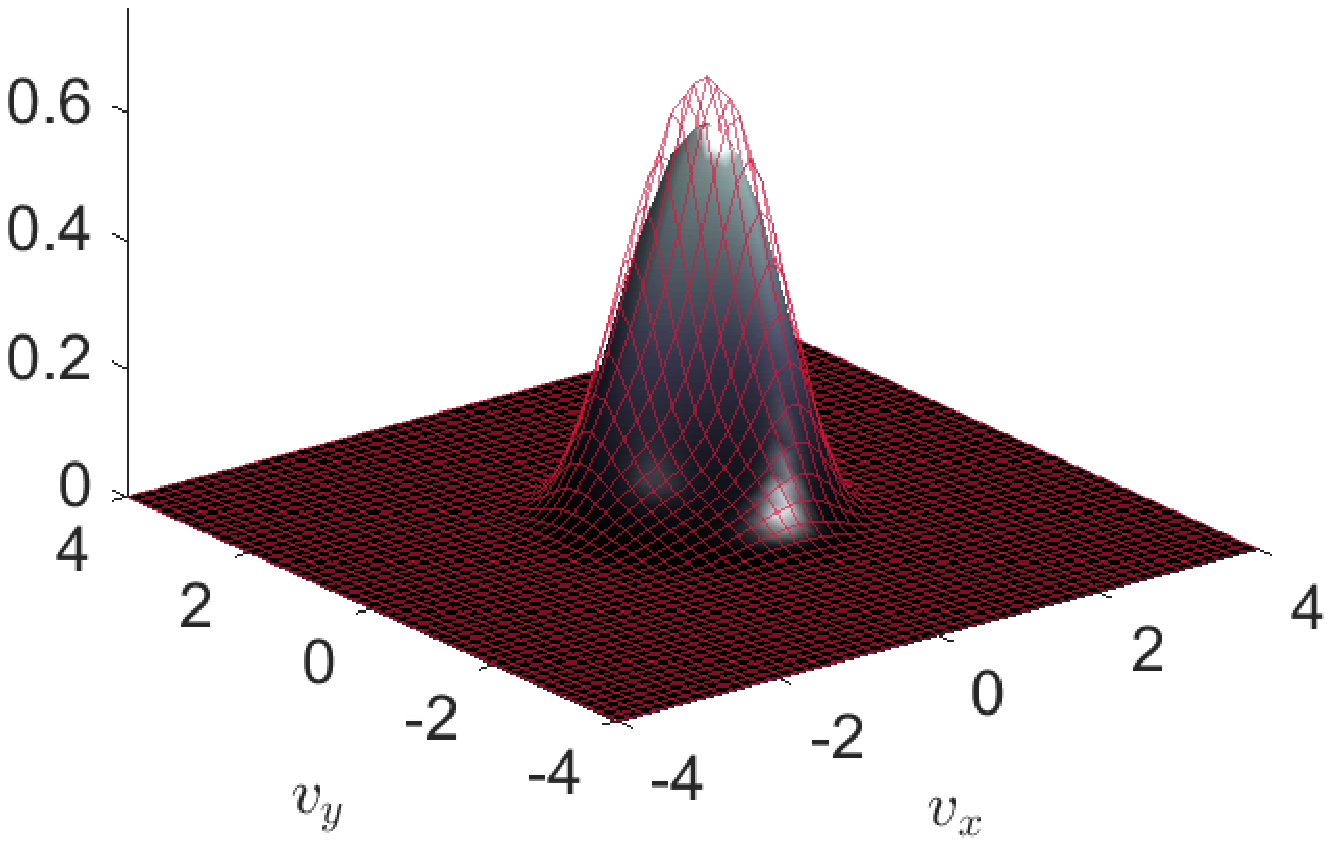}}
\hspace{-2em}
\subfigure[$D=0.2$, $\mu_x=\mu_y = 1$]{
\includegraphics[scale=0.5]{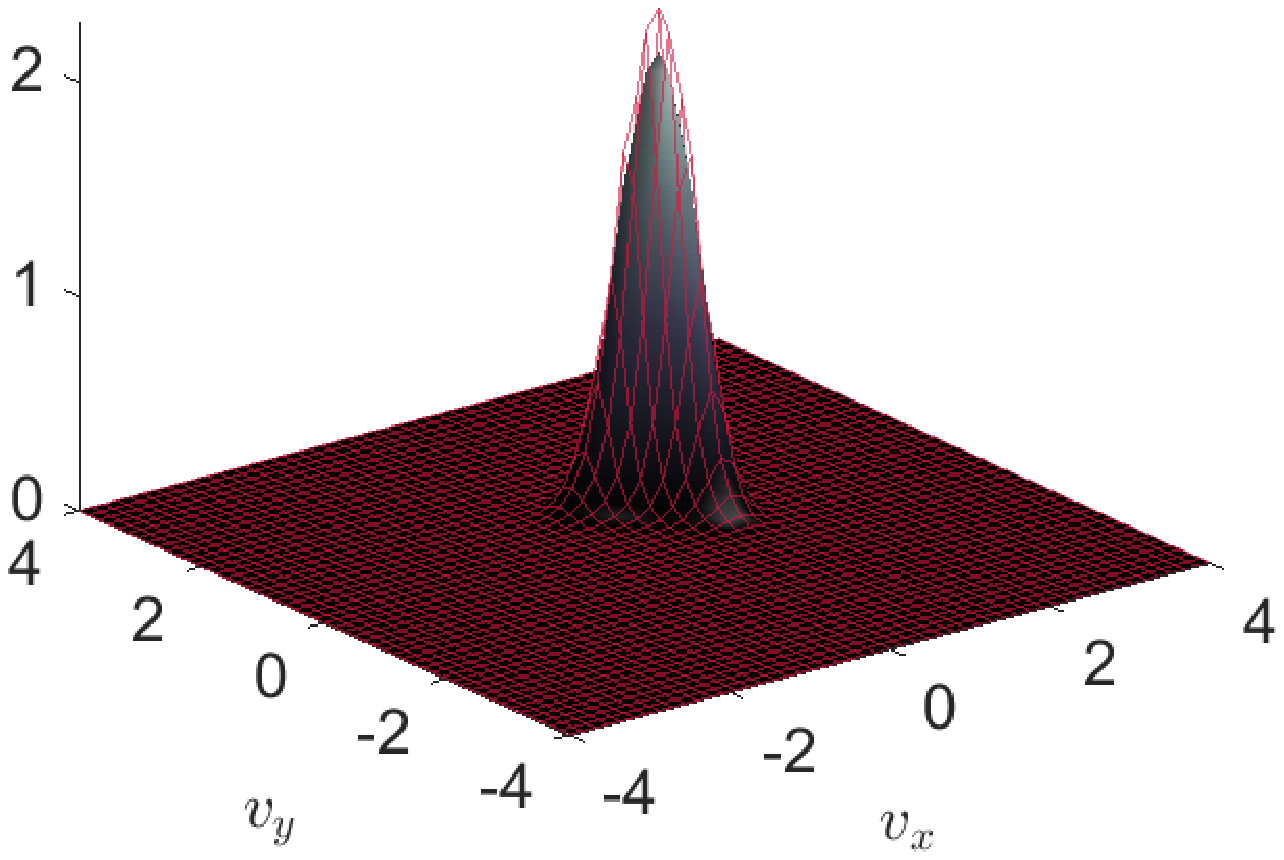}}\\
\subfigure[$D=0.8$, $\mu_x=0$, $\mu_y = 1$]{
\includegraphics[scale=0.5]{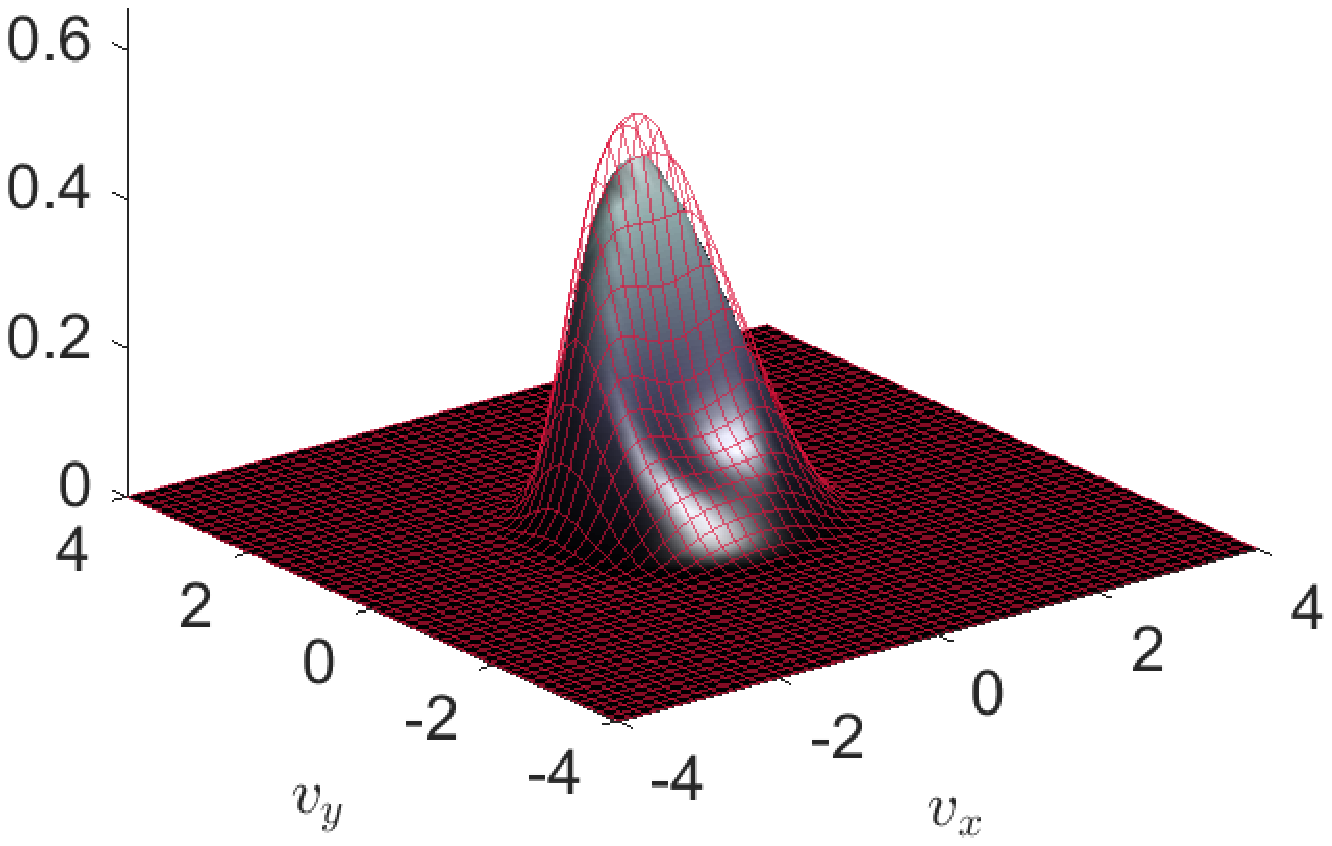}}
\hspace{-2em}
\subfigure[$D=0.2$, $\mu_x=0$, $\mu_y = 1$]{
\includegraphics[scale=0.5]{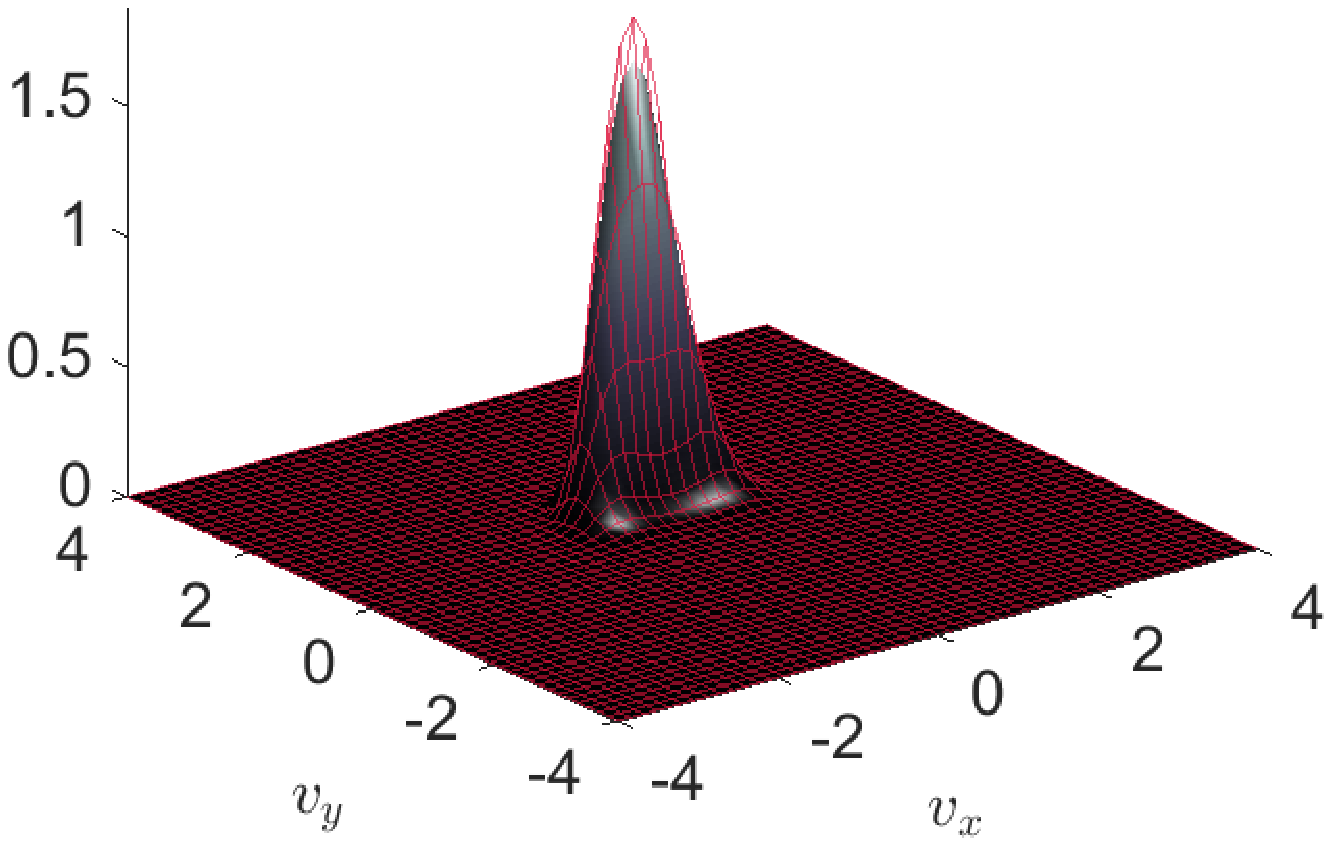}}
\caption{Example 3. Expected solutions at time $T=100$ of the 2D swarming model \eqref{eq:swarming} with initial uncertain distribution of mass $\rho(\theta) = 1+0.5\theta$, $\theta\sim \mathcal U([-1,1])$ and several background distributions \eqref{eq:back_swarming} obtained through the structure preserving stochastic Galerkin method $h=0,\dots,10$. Uniform grid for the velocity domain $[-4,4]\times[-4,4]$ with $N=51$ gridpoints in both directions, time integration with second order semi-implicit method $\Delta t = O(\Delta v)$. We visualize the upper confidence band through the red mesh.   }
\label{fig:2D}
\end{figure}

The extension of the presented structure preserving methods to the multidimensional case has been established in \cite{PZ1,PZ2}. The idea is to apply a structure preserving scheme to each dimension of the stochastic Galerkin projections 
\be\label{eq:2Dh}
\partial_t \hat f_h(v,t) = \nabla_v \cdot \left[\alpha v(|v|^2-1)\hat f_h(v,t) + (v-u_g)\hat f_h(v,t) + D\nabla_v \hat f_h(v,t) \right], \quad h = 0,\dots,M
\ee
with initial distribution $\hat f_h(v,0) = \mathbb E[f(\theta,v,0)\Phi_h(\theta)]$. In Figure \ref{fig:2D} we present the large time distributions for the choices of diffusion $D=0.2$, $D=0.8$ and three configuration of the background distribution. We applied the $SP_G$ scheme to \eqref{eq:2Dh} for $h = 0,\dots,5$ to obtain the approximation of expected distribution and of the variance. The initial distribution is here such that $\int_V f(\theta,v,0)dv = 1 +\frac{1}{2}\theta $, $\theta\sim \mathcal U([-1,1])$.

A uniform grid over $[-4,4]\times [-4,4]$ with $N = 51$ gridpoints in each direction has been considered. The  integration over the time interval $[0,100]$ has been performed taking advantage of the second order semi-implicit method with $\Delta t = O(\Delta v)$. The surfaces represent the expected solution whereas the red grids represent the upper confidence band that may be computed as $\mathbb E[f]+\sqrt{\textrm{Var}(f)}$. {The provided confidence bands gives an indication on the regions where $f(\theta,v,t)$ is more affected by the action of the uncertainty $\theta\in I_\Theta$.  }

We can clearly observe the influence of the background in shaping the large time distribution of the problem, which is steered towards the background mean. The computed confidence bands, furthermore, make clear how the behavior is stable under the action of initial uncertainties. 
%

\section*{Conclusion}
We studied the application of structure preserving type schemes to the stochastic Galerkin approximation of Fokker-Planck equations with uncertain initial distribution and background interactions. The developed methods are capable to preserve the stationary state of the problem with arbitrary accuracy and define nonnegative expected solutions under suitable time step restrictions. Both  explicit and semi-implicit type time integrations have been taken into account. Furthermore, we have proven discrete relative entropy dissipation property for the derived scheme for each projection of the original model. Several applications to prototype problems in socio-economic and life sciences have been proposed both in case of fixed and evolving background distribution together with the extension of the method to the multidimensional case. {Extensions of the scheme to the uncertain background case and to the case of vanishing diffusion are under study both in the deterministic and uncertain setting. }
\newpage
\section*{Acknowledgments}
This work has been written within the activities of GNFM (Gruppo Nazionale per la Fisica Mate\-matica) of INdAM (Istituto Nazionale di Alta Matematica), Italy. The author acknowledges partial support from the Excellence Project \textrm{CUP: E11G18000350001} of the Department of Mathematical Sciences ``G. L. Lagrange'', Politecnico di Torino, Italy. 


\begin{thebibliography}{99}

\bibitem{APZa}
G. Albi, L. Pareschi, and M. Zanella. Uncertainty quantification in control problems for flocking models, \emph{Math. Probl. Eng.}, 2015: 14 pp, 2015.

\bibitem{APTZ}
G. Albi, L. Pareschi, G. Toscani, and M. Zanella. Recent advances in opinion modeling: Control and social influence. In N. Bellomo, P. Degond, E. Tadmor (eds.) \emph{Active Particles, Volume 1}.  Modeling and Simulation in Science, Engineering and Technology. Birkh{\"a}user, Cham, 2017. 

\bibitem{BFR}
S. Boscarino, F. Filbet, and G. Russo. High order semi-implicit schemes for time dependent partial differential equations, \emph{J. Sci. Comput.}, 68: 975--1001, 2016. 

\bibitem{BMS}
L. Boudin, R. Monaco, and F. Salvarani. Kinetic model for multidimensional opinion formation. \emph{Phys. Rev. E}, 81(3): 036109, 2010. 

\bibitem{BCCD}
A. B. T. Barbaro, J. A. Ca\~{n}izo, J. A. Carrillo, and P. Degond. Phase transitions in a kinetic flocking models of Cucker-Smale type, \emph{Multiscale Model. Simul.}, 14.3 (2016), 1063--1088. 

\bibitem{CZ}
J. A. Carrillo, and M. Zanella. Monte Carlo gPC methods for diffusive kinetic flocking models with uncertainties. Preprint \texttt{arXiv:1902.04518}, 2019. 

\bibitem{CPZ}
J. A. Carrillo, L. Pareschi, and M. Zanella. Particle based gPC methods for mean-field models of swarming with uncertainty, \emph{Commun. Comput. Phys.}, 25(2): 508--531, 2019. 

\bibitem{CCP}
J. A. Carrillo, Y.-P. Choi, L. Pareschi. Structure preserving schemes for the continuum Kuramoto model: phase transitions, \emph{J. Comput. Phys.}, 17(1): 233--258, 2019. 

\bibitem{CIP}
C. Cercignani, R. Illner, and M. Pulvirenti. \emph{The Mathematical Theory of Dilute Gases}, Applied Mathematical Sciences vol. 106, Springer--Verlag, 1994. 

\bibitem{DB}
G. Dahlquist, and A. Bjorck. \emph{Numerical Methods in Scientific Computing: Volume 1}, SIAM, 2008. 

\bibitem{DPL}
B. Despr\'es G. Po{\"e}tte, and D. Lucor. Robust uncertainty propagation in systems of conservation laws with the entropy closure method. In \emph{Uncertainty Quantification in Computational Fluid Dynamics}, Lecture Notes in Computational Science and Engineering 92, 105--149, 2010.

\bibitem{DP15}
G. Dimarco, and L. Pareschi. Numerical methods for kinetic equations, \emph{Acta Numerica}, 23 (2014), 369--520.

\bibitem{DPZ}
G. Dimarco, L. Pareschi, and M. Zanella. Uncertainty quantification for kinetic models in socio--economic and life sciences. In \emph{Uncertainty Quantification for Hyperbolic and Kinetic Equations}, SEMA--SIMAI Springer Series, vol. 14, 2017. 

\bibitem{FPTT}
G. Furioli, A. Pulvirenti, E. Terraneo, and G. Toscani. Fokker-Planck equations in the modelling of socio-economic phenomena, \emph{Math. Mod. Meth. Appl. Sci.}, 27(1): 115--158, 2017. 

\bibitem{GT}
S. Gualandi, and G. Toscani. Pare tails in socio-economic phenomena: a kinetic description, \emph{Economics}, 12(2018-31): 1--17, 2018. 

{
\bibitem{HJJ}
S.-Y. Ha, S. Jin, and J. Jung. A local sensitivity analysis for the kinetic Cucker-Smale model with random inputs. \emph{J. Diff. Eqn.}, 265: 3618--3649, 2018. 
}

\bibitem{HJ}
J. Hu, and S. Jin. Uncertainty Quantification for Kinetic Equations. In \emph{Uncertainty Quantification for Kinetic and Hyperbolic Equations}, SEMA-SIMAI Springer Series, eds. S. Jin and L. Pareschi, vol. 14. 

\bibitem{HJX}
 J. Hu, S. Jin, and D. Xiu. A stochastic Galerkin method for Hamilton-Jacobi equations with uncertainty. \emph{SIAM J. Sci. Comp.}, 37, A2246-A2269, 2015.
 
 \bibitem{JP}
 S. Jin, and L. Pareschi. \emph{Uncertainty Quantification for Hyperbolic and Kinetic Equations}, SEMA SIMAI Springer Series, vol. 14, Springerm Cham, 2017.

\bibitem{JXZ}
S. Jin, D. Xiu, and X. Zhu. A well-balanced stochastic Galerkin method for scalar hyperbolic balance laws with random inputs, \emph{J. Sci. Comput.}, 67.3 (2016) ,1198--1218.

\bibitem{LZ}
N. Loy, and M. Zanella. Structure preserving schemes for nonlinear Fokker-Planck equations with anisotropic diffusion. Preprint \texttt{arXiv:1905.02970}, 2019. 

{
\bibitem{MJT}
D. Matthes, A. J{\"u}ngel, and G. Toscani. Convex Sobolev inequalities derived from entropy dissipation,  \emph{Arch. Ration. Mech. An.}, 199(2): 563--596, 2011. 
}

\bibitem{PT}
L. Pareschi, and G. Toscani. \emph{Interacting Multiagent Systems: Kinetic Equations and Monte Carlo Methods}, Oxford University Press, 2013. 

\bibitem{PT2}
L. Pareschi, and G. Toscani. Wealth distribution and collective knowledge. A Boltzmann approach, \emph{Phil. Trans. R. Soc. A}, 372, 20130396, 2014. 


\bibitem{PVZ}
L. Pareschi, P. Vellucci, and M. Zanella. Kinetic models of collective decision-making in the presence of equality bias, \emph{Physica A}, 467: 201--2017, 2017.


\bibitem{PZ1}
L. Pareschi, and M. Zanella. Structure-preserving schemes for nonlinear Fokker-Planck equations and applications. \emph{J. Sci. Comput.}, 74(3): 1575--1600, 2018. 

\bibitem{PZ2}
L. Pareschi, and M. Zanella. Structure preserving schemes for mean-field equations of collective behavior. In C. Klingenberg and M. Westdickenberg (eds.) \emph{Theory, Numerics and Applications of Hyperbolic Problems II. HYP 2016.} Springer Proceedings in Mathematics \& Statistics, 237, pp. 405--421, Springer, Cham. 

\bibitem{Pav}
G. A. Pavliotis. \emph{Stochastic Processes and Applications. Diffusion Processes, the Fokker-Planck and Langevin Equations}, Springer-Verlag New York, 2014. 

\bibitem{PreziosiTosin}
L. Preziosi, and A. Tosin. Multiphase modelling of tumor growth and extracellular matrix interaction: mathematical tools and applications, \emph{J. Math. Biol.}, 58: 625, 2009. 

\bibitem{Risken}
H. Risken, T. Frank. \emph{The Fokker-Planck Equation. Methods of Solution and Applications}, Springer-Verlag Berlin Heidelberg, 1996. 

{
\bibitem{SHJ}
R. Shu, J. Hu, and S. Jin. A stochastic Galerkin method for the Boltzmann equation with multi-dimensional random inputs using sparse wavelet bases, \emph{Numer. Math. Theory. Me.}, 10(2): 465--488, 2017. 
}

{
\bibitem{ToscQA}
G. Toscani. Entropy production and the rate of convergence to equilibrium for the Fokker-Planck equation. \emph{Quart. Appl. Math.}, 57(3):521--541, 1999. 
}

\bibitem{T}
G. Toscani. Kinetic models of opinion formation, \emph{Commun. Math. Sci.} 4.3 (2006), 481--496. 

\bibitem{TZ}
A. Tosin, and M. Zanella. Boltzmann-type models with uncertain binary interactions, \emph{Commun. Math. Sci.}, 16(4): 962--984, 2018. 

\bibitem{X}
D. Xiu, \emph{Numerical Methods for Stochastic Computations}, Princeton University Press, 2010.

\bibitem{XK}
D. Xiu, and G. E. Karniadakis. The Wiener--Askey polynomial chaos for stochastic differential equations, \emph{SIAM J. Sci. Comput.}, 24.2 (2002), 614--644.

\bibitem{ZJ} 
Y. Zhu, and S. Jin. The Vlasov-Poisson-Fokker-Planck system with uncertainty and a one-dimensional asymptotic-preserving method, \emph{Multiscale Model. Simul.},  15(4), 1502--1529, 2017.

\end{thebibliography}
\end{document}